\documentclass[reqno]{amsart}

\usepackage{amssymb}
\usepackage{amsfonts}
\usepackage{txfonts}
\usepackage{dsfont}

\title{Non-naturally reductive Einstein metrics on normal homogeneous Einstein manifolds}
\author{Zaili Yan$^{1}$ \and Shaoqiang Deng$^2$}
\address[Zaili Yan ]{Department of Mathematics, Ningbo University, Ningbo, Zhejiang Province, 315211,  People's Republic of China}
\email[]{yanzaili@nbu.edu.cn}
\address[Shaoqiang Deng]{School of Mathematical Sciences and LPMC,
Nankai University, Tianjin 300071, People's Republic of China}
\email[]{dengsq@nankai.edu.cn}
\thanks{$^1$Z. Yan is supported by NSFC (no. 11626134, 11401425) and  K.C. Wong Magna Fund in Ningbo University.}
\thanks{$^2$S. Deng is supported  by NSFC (no.  11671212, 51535008)  of China.}
\date{}

\newtheorem{thm}{Theorem}[section]
\newtheorem{prop}[thm]{Proposition}
\newtheorem{lem}[thm]{Lemma}

\theoremstyle{definition}
\newtheorem{defn}[thm]{Definition}

\newtheorem{rem}[thm]{Remark}

\begin{document}

\maketitle

\begin{abstract}
It is an important problem in differential geometry to find non-naturally reductive homogeneous Einstein metrics on homogeneous manifolds. In this paper, we consider this problem for   some coset spaces of compact simple Lie groups. A new  method to construct  invariant  non-naturally reductive Einstein metrics on normal homogeneous Einstein manifolds is presented.  In particular, we show that on the standard homogeneous Einstein manifolds, except for some special cases, there exist plenty of such metrics.   A further interesting result of this paper is that  on some  compact semisimple Lie groups, there exist a large  number of left invariant non-naturally reductive Einstein metrics which are not  product metrics.

\medskip
\textbf{Mathematics Subject Classification (2010)}: 53C25,  53C35, 53C30.

\medskip
\textbf{Key words}:  Einstein metrics,  Riemannian submersion,  naturally reductive metrics, standard homogeneous Einstein manifolds

\end{abstract}

\section{Introduction}

The study of Einstein metrics has been one of the central problems in Riemannian geometry.
Recall that a connected  Riemannian manifold $(M,g)$ is called Einstein if there exists a constant $c$ such that  $\mathrm{Ric}(g)=cg$, where  $\mathrm{Ric}(g)$ is the Ricci tesnor of $(M,g)$.  In general, the related problems in this field are rather involved and difficult. For example, till now a sufficient and necessary condition for a manifold to admit an Einstein metric is still unknown. As another remarkable open problem, it has been a long standing problem whether there is a nonstandard Einstein metric on the $4$-sphere $S^4$, see for example \cite{yau}.
This problem particularly reveals the fact that finding new examples of Einstein metrics
is essential in this topic.

 Although in the homogeneous case many beautiful results have been established, a complete classification of homogeneous Einstein manifolds still seems to be unreachable.
Even if in the  compact case, the classification has only been achieved for  spheres, normal homogeneous spaces and  naturally reductive metrics; see \cite{dz79,wz85}.
See also \cite{bohm04,bk06,bwz04,wz86,wolf68} for some important and interesting results on the existence (or non-existence) of homogeneous or inhomogeneous Einstein metrics on some special manifolds. Meanwhile, in the literature there are some excellent surveys of the development of this field, see for example
   \cite{besse87book, nrs07,wang99}.

 The method of Riemannian submersion is an important tool to construct new examples of
 Einstein metrics, and it has been applied to obtain many interesting existence results; see
 Chapter 9 of \cite{besse87book} and some results in \cite{ara10,ara11,dk08}.
  Let $G/H$ be a compact connected homogeneous space, and $\mathfrak{g}=\mathfrak{h}+\mathfrak{m}$ a reductive decomposition of $\mathfrak{g}$, where $\mathfrak{g}$, $\mathfrak{h}$ denote the Lie algebras of $G$ and $H$ respectively, and $\mathfrak{m}$ is a subspace of $\mathfrak{g}$ such that
 $\mathrm{Ad}(H)(\mathfrak{m})\subset \mathfrak{m}$. Then there is a one-to-one correspondence
 between the $G$-invariant Riemannian metric on $G/H$ and the $\mathrm{Ad}(H)$-invariant inner product on $\mathfrak{m}$.
 Recall that an invariant metric on $G/H$ is called normal if the corresponding inner product on $\mathfrak{m}$ is the restriction of a bi-invariant inner product on $\mathfrak{g}$.   In particular, let $B$ denote the negative Killing form of $\mathfrak{g}$,
  and $g_{B}$  be the standard metric   on $G/H$  induced by $B|_{\mathfrak{m}}$. Then $g_{B}$ is normal. The coset space  $G/H$ is called a standard
 homogeneous Einstein manifold if the standard metric $g_{B}$ is Einstein. In \cite{wz85}, M. Wang and W. Ziller obtained a classification of  standard homogeneous Einstein manifolds $G/H$ with $G$ compact simple.
 Let $K/H\rightarrow G/H\rightarrow G/K$   be a Riemannian submersion with totally geodesic fibres.
 Assume that the standard metrics on   $G/H$ and  $G/K$ are Einstein, and there exists a  constant $c$ such that $B_{\bar{\mathfrak{k}}}=cB|_{\bar{\mathfrak{k}}}$,
 where $\bar{K}/\bar{H}$ is the corresponding (almost) effective quotient of $K/H$,
 and $B_{\bar{\mathfrak{k}}}$ is the negative Killing form of
 $\bar{\mathfrak{k}}=\mathrm{Lie}(\bar{K})$. Then besides the standard homogeneous Einstein metric, M. Wang and W. Ziller \cite{wz85} showed that there exists another (non-naturally reductive) homogeneous Einstein metric on $G/H$ except  some special cases;
 see Table XI of \cite{wz85} for a complete classification of the Riemannian submersions
 $K/H\rightarrow G/H\rightarrow G/K$.

 This paper is a continuation of our previous work \cite{YZ}.  Inspired by the ideas of Riemannian submersion of M. Wang and W. Ziller \cite{wz85,wz86}, we consider a family of invariant metrics on $G/H$  depending on two real parameters  associated to two Riemannian submersions
 $K/H\rightarrow G/H\rightarrow G/K$ and $L/H\rightarrow G/H\rightarrow G/L$. More precisely, given a basic quadruple
 $(G,L,K,H)$ (see Definition \ref{def3.1}), where the Lie algebra $\mathfrak{g}$ has a $B$-orthogonal decomposition
\begin{equation}\label{}
\mathfrak{g}=\mathfrak{l}+\mathfrak{p}
=\mathfrak{k}+\mathfrak{u}+\mathfrak{p}=\mathfrak{h}+\mathfrak{n}+\mathfrak{u}+\mathfrak{p},
\quad \mathfrak{m}=\mathfrak{n}+\mathfrak{u}+\mathfrak{p},
\end{equation}
where $\mathfrak{n}$, $\mathfrak{u}$, $\mathfrak{p}$ are the subspaces of $\mathfrak{k}$, $\mathfrak{l}$ and $\mathfrak{g}$ respectively, and $\mathfrak{k}=\mathrm{Lie(K)}$, $\mathfrak{l}=\mathrm{Lie(L)}$,
    we consider  $G$-invariant metrics  of the form
\begin{equation}\label{g(x,y)}
  \langle , \rangle=g_{(x,y)}=B|_{\mathfrak{n}}+xB|_{\mathfrak{u}}+yB|_{\mathfrak{p}}, \quad x,y\in \mathds{R}^{+},
\end{equation}
on the homogeneous space $G/H$. Our goal is to find out under what conditions there exist new Einstein metrics, and if so, to classify them.
It is clear that the invariant metric $g_{(1,y)}$ corresponds to the Riemannian submersion
  $L/H\rightarrow G/H\rightarrow G/L$, and the invariant metric $g_{(x,x)}$ corresponds to the Riemannian submersion
$K/H\rightarrow G/H\rightarrow G/K$.

Our first main theorem  is the following
\begin{thm}\label{main1}
Let $(G,L,K,H)$ be a  basic quadruple with $G$  compact simple. Suppose   the standard metrics on $G/L$, $G/K$, $G/H$ are Einstein. If  $H\neq \{e\}$, then $(G,L,K,H)$ must be one of the quadruples in Table A; If  $H=\{e\}$, then $(G,L,K,H)$ must be one of the quadruples in Table B.
\end{thm}

 Next we study the Ricci curvature of   $g_{(x,y)}$, and obtain a sufficient and necessary condition for $g_{(x,y)}$ to be Einstein; see Proposition \ref{prop4.4}. Then we prove
\begin{thm}\label{main2}
Let $(G,L,K,H)$ be one of the basic quadruples in Table A and Table B. Then besides the three homogeneous Einstein metrics associated to the Riemannian submersions $K/H\rightarrow G/H\rightarrow G/K$
and  $L/H\rightarrow G/H\rightarrow G/L$, there  always exists another Einstein metric on $G/H$
of the form $g_{(x,y)}$ with $x\neq 1, x\neq y$, except for the following three cases:
\begin{enumerate}
\item  Type A. 4 with $n_{1}=9m+1$, $n_{2}=n_{3}=2$, $k=2m$, $m\in \mathds{N}^{+}$, namely, the quadruples
\begin{equation}
  \bigg(\mathfrak{sp}(8m(9m+1)), (9m+1)\mathfrak{sp}(8m), 2(9m+1)\mathfrak{sp}(4m), 4(9m+1)\mathfrak{sp}(2m)\bigg).
\end{equation}
\item  Type A. 5:
 \begin{equation}
  \bigg(\mathfrak{e}_{6}, \mathfrak{so}(10)\oplus \mathds{R},
   \mathfrak{so}(8)\oplus \mathds{R}^{2}, \mathds{R}^{6}\bigg).
\end{equation}
\item Type B. 3 with  $n_{1}=n_{2}=2$, $k=1$, namely, the quadruple
\begin{equation}
  \bigg(\mathfrak{sp}(4), 2\mathfrak{sp}(2), 4\mathfrak{sp}(1), \{e\}\bigg).
\end{equation}
\end{enumerate}
 \end{thm}

As an application of Theorems \ref{main1} and  \ref{main2}, we obtain some new invariant Einstein metrics on some flag manifolds $G/T$, where $G=\mathrm{SU}(n), \mathrm{SO(2n)}$, or $E_{8}$, and $T$ is a maximal compact connected abelian subgroup of $G$.   Moreover, Table B provides  many new invariant Einstein metrics on compact simple Lie groups which are not naturally reductive. Finally, we prove the following
\begin{thm}
Let $n=p_{1}^{l_{1}}p_{2}^{l_{2}}\cdots p_{s}^{l_{s}}$ be a positive integer, where the $p_i$'s are prime numbers and $p_{i}\neq p_{j}$, when $i\neq j$.  Let $H$ be a compact connected simple Lie group and $G=H\times H\times \cdots \times H$ ($n$ times).
 Then  $G$  admits at least $(l_{1}+1)(l_{2}+1)\cdots (l_{s}+1)-2$ left invariant non-equivalent  non-naturally reductive Einstein metrics.
\end{thm}

\pagebreak
\begin{center}
  Table A: Standard quadruples with $G$ simple, $H\neq \{e\}$
\end{center}

\medskip
\begin{center}
\begin{tabular}{|c|c|c|c|c|c|}
  \hline
  Type A& $\mathfrak{g}$ & $\mathfrak{h}$ & $\mathfrak{l}$ & $\mathfrak{k}$ & Remarks \\
  \hline
  1 &$\mathfrak{su}(n_{1}n_{2}n_{3}k)$ & $s(n_{1}n_{2}n_{3}\mathfrak{u}(k))$&$s(n_{1}\mathfrak{u}(n_{2}n_{3}k))$&
  $s(n_{1}n_{2}\mathfrak{u}(n_{3}k))$&  $k\geq 1, n_{i}\geq 2$ \\
  \hline
  2 & $\mathfrak{so}(n_{1}n_{2}n_{3}k)$ & $n_{1}n_{2}n_{3}\mathfrak{so}(k)$ & $n_{1}\mathfrak{so}(n_{2}n_{3}k)$ &
   $n_{1}n_{2}\mathfrak{so}(n_{3}k)$ & $k\geq 2, n_{i}\geq 2$ \\
  \hline
  3 & $\mathfrak{so}(n_{1}n_{2}k)$ & $\oplus_{i=1}^{l}\mathfrak{h}_{i}$ & $n_{1}\mathfrak{so}(n_{2}k)$
  &$n_{1}n_{2}\mathfrak{so}(k)$ & $k\geq 3, n_{i}\geq 2$ \\
  \hline
  4 & $\mathfrak{sp}(n_{1}n_{2}n_{3}k)$ & $n_{1}n_{2}n_{3}\mathfrak{sp}(k)$ & $n_{1}\mathfrak{sp}(n_{2}n_{3}k)$
  & $n_{1}n_{2}\mathfrak{sp}(n_{3}k)$  & $k\geq 1, n_{i}\geq 2$ \\
  \hline
  5 & $\mathfrak{e}_{6}$ & $\mathds{R}^{6}$ & $\mathfrak{so}(10)\oplus \mathds{R}$
   & $\mathfrak{so}(8)\oplus \mathds{R}^{2}$ &  \\
  \hline
  6 & $\mathfrak{e}_{7}$ & $7\mathfrak{su}(2)$ & $\mathfrak{so}(12)\oplus \mathfrak{su}(2)$
  & $\mathfrak{so}(8)\oplus 3\mathfrak{su}(2)$ &  \\
  \hline
  7 & $\mathfrak{e}_{8}$ & $\mathds{R}^{8}$ & $\mathfrak{so}(16)$ & $2\mathfrak{so}(8)$ &  \\
  \hline
  8 & $\mathfrak{e}_{8}$ & $\mathds{R}^{8}$ & $\mathfrak{so}(16)$ & $8\mathfrak{su}(2)$ &  \\
  \hline
  9 & $\mathfrak{e}_{8}$ & $\mathds{R}^{8}$ & $2\mathfrak{so}(8)$ & $8\mathfrak{su}(2)$ &  \\
  \hline
  10 & $\mathfrak{e}_{8}$ & $8\mathfrak{su}(2)$ & $\mathfrak{so}(16)$ & $2\mathfrak{so}(8)$ &  \\
  \hline
  11 & $\mathfrak{e}_{8}$ & $2\mathfrak{su}(3)$ & $\mathfrak{so}(16)$ & $2\mathfrak{so}(8)$ &  \\
  \hline
\end{tabular}
\end{center}

\medskip

\begin{center}
  Table B: Standard quadruples with $G$ simple, $H=\{e\}$
\end{center}

\medskip

\begin{center}
\begin{tabular}{|c|c|c|c|c|}
  \hline
  Type B& $\mathfrak{g}$ & $\mathfrak{l}$ & $\mathfrak{k}$ & Remarks \\
  \hline
  1 & $\mathfrak{so}(n_{1}n_{2}k)$ & $n_{1}\mathfrak{so}(n_{2}k)$   & $n_{1}n_{2}\mathfrak{so}(k)$
  & $k\geq 3, n_{i}\geq 2$ \\
  \hline
  2 & $\mathfrak{so}(nk)$ & $n\mathfrak{so}(k)$ & $\oplus_{i=1}^{l}\mathfrak{h}_{i}$ & $k\geq 3, n\geq 2$ \\
  \hline
  3 & $\mathfrak{sp}(n_{1}n_{2}k)$ & $n_{1}\mathfrak{sp}(n_{2}k)$   & $n_{1}n_{2}\mathfrak{sp}(k)$
  & $k\geq 1, n_{i}\geq 2$ \\
  \hline
  4 & $\mathfrak{so}(8)$ & $\mathfrak{so}(7)$ & $\mathfrak{g}_{2}$ &  \\
  \hline
  5 & $\mathfrak{f}_{4}$ & $\mathfrak{so}(9)$ & $\mathfrak{so}(8)$ &  \\
  \hline
  6 & $\mathfrak{e}_{6}$ & $3\mathfrak{su}(3)$ & $3\mathfrak{so}(3)$ &  \\
  \hline
  7 & $\mathfrak{e}_{7}$ & $\mathfrak{su}(8)$ & $\mathfrak{so}(8)$ &  \\
  \hline
  8 & $\mathfrak{e}_{8}$  & $\mathfrak{so}(16)$ & $2\mathfrak{so}(8)$ &  \\
  \hline
  9 & $\mathfrak{e}_{8}$  & $\mathfrak{so}(16)$ & $\mathfrak{so}(9)$ &  \\
  \hline
  10 & $\mathfrak{e}_{8}$ & $\mathfrak{so}(16)$ & $8\mathfrak{su}(2)$ &  \\
  \hline
  11 & $\mathfrak{e}_{8}$  & $\mathfrak{so}(16)$ & $2\mathfrak{so}(5)$ &  \\
  \hline
  12 & $\mathfrak{e}_{8}$ & $\mathfrak{so}(16)$ & $2\mathfrak{su}(3)$ &  \\
  \hline
  13 & $\mathfrak{e}_{8}$ & $\mathfrak{su}(9)$ & $\mathfrak{so}(9)$ &  \\
  \hline
  14 & $\mathfrak{e}_{8}$ & $\mathfrak{su}(9)$ & $2\mathfrak{su}(3)$ &  \\
  \hline
  15 & $\mathfrak{e}_{8}$ & $2\mathfrak{so}(8)$ & $8\mathfrak{su}(2)$ &  \\
  \hline
  16 & $\mathfrak{e}_{8}$ & $2\mathfrak{so}(8)$ & $2\mathfrak{su}(3)$ &  \\
  \hline
  17 & $\mathfrak{e}_{8}$ & $2\mathfrak{su}(5)$ & $2\mathfrak{so}(5)$ &  \\
  \hline
  18 & $\mathfrak{e}_{8}$  & $4\mathfrak{su}(3)$ & $4\mathfrak{so}(3)$ &  \\
  \hline
\end{tabular}
\end{center}

\begin{rem}
In this paper, $nG$ means $G\times G\times \cdots \times G $ ($n$ times).
\end{rem}

In Section 2, we survey some results on homogeneous Einstein metrics. In particular, we recall some results of M. Wang and W. Ziller on naturally reductive and non-naturally reductive Einstein metrics. In Section 3, we give the definition and classification of standard quadruples. Section 4 is devoted to the calculation of Ricci curvature of the related coset spaces. The main results of this paper are proved in Section 5. To make the main proofs of the paper more concise, we  collect some repetitive case by case calculations in Section 5 as two appendixes.

\section{Naturally reductive and Non-naturally reductive Einstein metrics}
In this section, we recall some results on   naturally reductive and non-naturally reductive Einstein metrics,
 for details, see \cite{ams12,dz79}.

Let $(M,g)$ be a connected Riemannian manifold and $I(M,g)$ the full group of  isometries of $M$. Given a Lie subgroup $G$ of $I(M, g)$, the Riemannian manifold $(M,g)$ is said to be $G$-homogeneous if  $G$  acts transitively on $M$. For a $G$-homogeneous Riemannian manifold, we fix a point $o\in M$ and identify $M$ with $G/H$, where
$H$ is the isotropy subgroup of $G$ at $o$. Let $\mathfrak{g}, \mathfrak{h}$ be the Lie algebras of $G$ and $H$ respectively. Then $\mathfrak{g}$ has a
reductive decomposition $\mathfrak{g}=\mathfrak{h}+\mathfrak{m}$ (direct sum of subspaces), where $\mathfrak{m}$ is a subspace of $\mathfrak{g}$ satisfying
 $\mathrm{Ad}(H)(\mathfrak{m})\subset \mathfrak{m}$. Then one can identify $\mathfrak{m}$  with $T_{o}M$ through the map
 $$X\rightarrow \frac{d}{dt}|_{t=0}(\exp(tX)\cdot o).$$
  In this case, one can pull back the inner product $g_{o}$ on $T_{o}M$ to get an inner product on $\mathfrak{m}$, denoted by
 $\langle ,\rangle$. Given $X\in \mathfrak{g}$, we  denote by $X_{\mathfrak{m}}$ the $\mathfrak{m}$-component of $X$. Then a homogeneous Riemannian metric on $M$ is said to be naturally reductive if there exists a transitive subgroup $G$ and $\mathfrak{m}$ as above such that
 $$\langle [Z,X]_{\mathfrak{m}},Y\rangle +\langle X,[Z,Y]_{\mathfrak{m}}\rangle=0, \quad \forall X, Y, Z \in \mathfrak{m}.$$

 In \cite{dz79},  D'Atri and  Ziller  investigated naturally reductive metrics among the left invariant metrics on compact Lie groups, and give a complete description of this type of  metrics on simple Lie groups. Now we recall the main results of them.

 Let $G$ be a compact connected semisimple Lie group,  and $H$ a closed subgroup of $G$. Denote by  $B$ the negative of the Killing form of $\mathfrak{g}$. Then $B$
 is an $\mathrm{Ad}(G)$-invariant inner product on $\mathfrak{g}$. Let $\mathfrak{m}$ be the orthogonal complement of $\mathfrak{h}$ with respect to $B$.
 Then we have $$\mathfrak{g}=\mathfrak{h}\oplus\mathfrak{m}, \quad \mathrm{Ad}(H)(\mathfrak{m})\subset \mathfrak{m}.$$
 Let $$\mathfrak{h}=\mathfrak{h}_{0}\oplus \mathfrak{h}_{1}\oplus \mathfrak{h}_{2}\oplus \cdots \oplus \mathfrak{h}_{p} $$ be the decomposition of $\mathfrak{h}$ into ideals,
  where $\mathfrak{h}_{0}$ is the center of $\mathfrak{h}$ and $\mathfrak{h}_{i}$ $(i=1,\ldots,p)$ are simple ideals of $\mathfrak{h}$.
 Let $A_{0}|_{\mathfrak{h}_{0}}$ be an arbitrary metric on $\mathfrak{h}_{0}$.
 \begin{thm}[\cite{dz79}] \label{nat-thm}
 Keep the notation as above. Then a left invariant metric on $G$ of the form
 \begin{equation}\label{naturally reductive}
   \langle,\rangle= xB|_{m}+A_{0}|_{\mathfrak{h}_{0}}+ u_{1}B|_{\mathfrak{h}_{1}}+\cdots+u_{p}B|_{\mathfrak{h}_{p}},
 \end{equation}
 where $x,u_{1},\ldots,u_{p}$ are positive real numbers, must be naturally reductive with respect to $G\times H$, where $G\times H$ acts on $G$ by $(g,h)y=gyh^{-1}$.

 Moveover, if a left invariant metric $\langle ,\rangle$ on a compact simple Lie group $G$ is naturally reductive, then there exists a closed subgroup $H$ of $G$
 such that the metric $\langle , \rangle$ is given by the form \eqref{naturally reductive}.
\end{thm}
 Based on the above theorem, D'Atri and  Ziller \cite{dz79} obtained a large number of naturally reductive Einstein metrics on compact simple Lie groups.

  Now we recall some results of Wang and Ziller. Let $(G/H,g_{B})$ be a compact connected homogeneous  space with the reductive decomposition
  $\mathfrak{g}=\mathfrak{h}+\mathfrak{m}$.  Denote by $\chi$ the isotropy representation of $\mathfrak{h}$ on $\mathfrak{m}$. Let $C_{\chi,\mathfrak{m}}$ be the Casimir operator defined by $-\mathop{\sum}\limits_{i}(\mathrm{ad}(X_{i})\mathrm{ad}(X_{i}))|_{\mathfrak{m}}$, where
 $\{X_{i}\}$ is a $B$-orthonormal basis of $\mathfrak{h}$.  Wang and Ziller obtained a sufficient and necessary condition for $g_{B}$ to be Einstein.
 \begin{thm}[\cite{wz85}]
 The standard homogeneous metric $g_{B}$ on $G/H$ is Einstein if and only if there exists a  constant $a$ such that $C_{\chi, \mathfrak{m}}=a\,\mathrm{id}$, where $\mathrm{id}$ denotes the identity transformation.
 \end{thm}
 Based on this theorem and some deep results on   representation theory, Wang and Ziller give a complete classification of standard Einstein manifolds $G/H$ for any compact simple Lie group $G$.

 Given a subalgebra $\mathfrak{h}$ of $\mathfrak{g}$, one can consider  the metric $g_{t}=B|_{\mathfrak{h}}+tB|_{\mathfrak{m}}, \quad t>0, $
as a left invariant metric on $G$. Clearly, $g_{t}$ is naturally reductive. If $t=1$, then $g_{t}$ is Einstein since it is bi-invariant. G. Jensen \cite{jensen73} first
studied the Einstein metrics of the form  $g_{t}$, where  $t\neq 1$. Subsequently  D'Atri and  Ziller proved the following
\begin{thm}[\cite{dz79}]\label{standard}
 Suppose $\mathfrak{h}$ is not an ideal in $\mathfrak{g}$. Then there exists a unique $t_0\neq 1$ with  $g_{t_0}$  Einstein if and only if the standard metric
on $G/H$ is  Einstein and there exists a constant $c$ such that $B_{\mathfrak{h}}=cB|\mathfrak{h}$.
Furthermore, in this case, we have $t_0>1$ and $g_{t_0}$ must be normal homogeneous with respect to $G\times H$.
In particular, if $\mathfrak{h}$ is abelian, then $t=1$ is the only real number such that $g_t$ is an Einstein metric.
\end{thm}

Many non-naturally reductive Einstein metrics can be constructed by Riemannian submersions. We recall the following result, see \cite{wz85} and page 255 of \cite{besse87book}.
\begin{thm}[\cite{besse87book}]
Let $F\rightarrow M\rightarrow B$ be a Riemannian submersion with totally geodesic fibres. Assume that the metrics
on $F$, $M$ and $B$ are Einstein with Einstein constant $r_{F}$, $r_{M}$, $r_{B}$ respectively, and $r_{F}>0$.
Furthermore, suppose $M$ is not locally  a Riemannian product of $F$ and $B$. Then the metric $g_{t}$ obtained by scaling the metric on $M$ in the direction of $F$ by a factor $t\neq 1, t>0$ is Einstein if and only if
$r_{F}\neq \frac{1}{2}r_{B}$.
\end{thm}

Applying this theorem to the homogeneous case,
 Wang and Ziller obtained a great number of non-naturally reductive
Einstein metrics on standard homogeneous Einstein manifolds. In fact, in \cite{wz85}, they give a complete classification of the Riemannian submersions
 $K/H\rightarrow G/H\rightarrow G/K$, where $G$ is compact simple, such that the standard metrics
on $G/H$, $G/K$ are Einstein, and   there exists a constant $c>0$ such that $B_{\bar{\mathfrak{k}}}=cB|_{\bar{\mathfrak{k}}}$,
where $\bar{K}/\bar{H}$ is the corresponding effective (almost) quotient of $K/H$,
$B$ and $B_{\bar{\mathfrak{k}}}$  are the negative Killing forms of
$\mathfrak{g}=\mathrm{Lie}(G)$ and $\bar{\mathfrak{k}}=\mathrm{Lie}(\bar{K})$, respectively.

Up to now, most known examples of Einstein metrics on compact simple Lie groups are naturally reductive; see \cite{ams12,jensen73,mujtaba12,pope10}.
The problem of finding  left invariant Einstein metrics on compact Lie groups which are not naturally reductive is more difficult, and    is stressed by J.E. D'Atri and W. Ziller in \cite{dz79}.
In 1994, Mori initiated the study of this problem. Mori  showed that there exists   non-naturally reductive Einstein metrics on the Lie group $\mathrm{SU}(n)$ with $n\geq 6$
 by using the method of Riemannian submersions \cite{mori94}. Later, in \cite{ams12}, the authors established the existence of non-naturally reductive Einstein metrics on the compact simple Lie groups $\mathrm{SO}(n)$
  with $n\geq 11$, $\mathrm{Sp}(n)$ with  $n\geq 3$, and the exceptional groups $E_{6}$, $E_{7}$ and $E_{8}$.  Recently, some non-naturally reductive Einstein metrics have been found on the compact simple
 Lie groups $\mathrm{SU}(3)$, $\mathrm{SO}(5)$, $G_{2}$ and $F_{4}$; see \cite{cl14,glp11}. We summarize the above results as following
\begin{thm}{\rm (\cite{mori94,ams12,cl14,glp11})}\quad
 The compact simple Lie groups $\mathrm{SU}(n)\, (n\geq 6)$, $\mathrm{SO}(n)\, (n \geq 11)$,
 $\mathrm{Sp}(n)\, (n \geq 3)$, $G_{2}$, $F_{4}$, $E_{6}$, $E_{7}$ and $E_{8}$
 admit non-naturally reductive Einstein metrics.
\end{thm}

 Up to now, it has been an open problem   whether there exists a left invariant non-naturally reductive Einstein metric on the compact simple Lie groups $\mathrm{SU}(n)$, with $n=4,5$, or $\mathrm{SO}(n)$, with  $ n=7,8,9,10$.

\section{Classification of standard quadruples}

Let $G$ be a compact semisimple connected Lie group, 
  and $H\subsetneqq K\subsetneqq L$ be three closed  proper  subgroups of $ G$ such that $G$ acts effectively on the  coset space $G/H$. We denote by
$\mathfrak{h}$, $\mathfrak{k}$, $\mathfrak{l}$, $\mathfrak{g}$ the Lie algebras of $H$, $K$, $L$, $G$, respectively,  and
$B_{\mathfrak{h}}$, $B_{\mathfrak{k}}$, $B_{\mathfrak{l}}$, $B$ the negative of the Killing forms of $\mathfrak{h}$, $\mathfrak{k}$, $\mathfrak{l}$, $\mathfrak{g}$, respectively.
Then $\mathfrak{g}$ has a $B$-orthogonal decomposition
$$\mathfrak{g}=\mathfrak{l}+\mathfrak{p}=\mathfrak{k}+\mathfrak{u}+\mathfrak{p}=\mathfrak{h}+\mathfrak{n}+\mathfrak{u}+\mathfrak{p},$$
where $\mathfrak{n}$, $\mathfrak{u}$, $\mathfrak{p}$ are the subspaces of $\mathfrak{k}$, $\mathfrak{l}$ and $\mathfrak{g}$ respectively.
Denote  $\mathfrak{m}=\mathfrak{n}+\mathfrak{u}+\mathfrak{p}$. Then it is easily seen that
\begin{equation*}
[\mathfrak{h},\mathfrak{n}]\subset \mathfrak{n},\quad [\mathfrak{h}+\mathfrak{n},\mathfrak{u}]\subset \mathfrak{u},
\quad [\mathfrak{h}+\mathfrak{n}+\mathfrak{u},\mathfrak{p}]\subset \mathfrak{p}.
\end{equation*}

Let $\chi_{\mathfrak{h},\mathfrak{n}}$, $\chi_{\mathfrak{h},\mathfrak{u}}$, $\chi_{\mathfrak{h},\mathfrak{p}}$,
$\chi_{\mathfrak{k},\mathfrak{u}}$, $\chi_{\mathfrak{k},\mathfrak{p}}$, $\chi_{\mathfrak{l},\mathfrak{p}}$ be the adjoint representation of $\mathfrak{h}$ on
$\mathfrak{n}$, $\mathfrak{u}$, $\mathfrak{p}$,  $\mathfrak{k}$ on $\mathfrak{u}$, $\mathfrak{p}$
and $\mathfrak{l}$ on $\mathfrak{p}$, respectively,
and $C_{\mathfrak{h},\mathfrak{n}}$, $C_{\mathfrak{h},\mathfrak{u}}$, $C_{\mathfrak{h},\mathfrak{p}}$,
$C_{\mathfrak{k},\mathfrak{u}}$, $C_{\mathfrak{k},\mathfrak{p}}$, $C_{\mathfrak{l},\mathfrak{p}}$ the corresponding Casimir operators defined by
$$C_{\mathfrak{h},\mathfrak{n}}=-\sum(\mathrm{ad}(h_{i})\mathrm{ad}(h_{i}))|_{\mathfrak{n}},$$  $$C_{\mathfrak{h},\mathfrak{u}}=-\sum(\mathrm{ad}(h_{i})\mathrm{ad}(h_{i}))|_{\mathfrak{u}}, $$
$$C_{\mathfrak{h},\mathfrak{p}}=-\sum(\mathrm{ad}(h_{i})\mathrm{ad}(h_{i}))|_{\mathfrak{p}},$$
$$C_{\mathfrak{k},\mathfrak{u}}=-\sum(\mathrm{ad}(k_{i})\mathrm{ad}(k_{i}))|_{\mathfrak{u}},$$
$$C_{\mathfrak{k},\mathfrak{p}}=-\sum(\mathrm{ad}(k_{i})\mathrm{ad}(k_{i}))|_{\mathfrak{p}},$$
$$C_{\mathfrak{l},\mathfrak{p}}=-\sum(\mathrm{ad}(l_{i})\mathrm{ad}(l_{i}))|_{\mathfrak{p}},$$
 where $\{h_{i}\}$,   $\{k_{i}\}$ and  $\{l_{i}\}$ are B-orthonormal basis of $\mathfrak{h}$, $\mathfrak{k}$, $\mathfrak{l}$, respectively.

 Note that even if $G$ is simple, $L/H$ and $K/H$ need not be effective, so we denote by $\bar{L}/\bar{H_{1}}$ and $\bar{K}/\bar{H_{2}}$
 the corresponding (almost) effective quotient.
\begin{defn}\label{def3.1}
Let the notation be as above.  A quadruple  $(G,L,K,H)$ is called  a basic  quadruple if it satisfies the following conditions:
\begin{enumerate}
\item $G$ is compact and acting effectively on $G/H$;
\item There exist constants $c_{1}$, $c_{2}> 0$ such that
$B_{\bar{\mathfrak{l}}}=c_{1}B|_{\bar{\mathfrak{l}}}$,
$B_{\bar{\mathfrak{k}}}=c_{2}B|_{\bar{\mathfrak{k}}}$.
\item There exist constants $h_{\mathfrak{n}}$, $h_{\mathfrak{u}}$, $h_{\mathfrak{p}}$, $k_{\mathfrak{u}}$, $k_{\mathfrak{p}}$, $l_{\mathfrak{p}}$ such that
$C_{\mathfrak{h},\mathfrak{n}}=h_{\mathfrak{n}}\mathrm{id}$,
$C_{\mathfrak{h},\mathfrak{u}}=h_{\mathfrak{u}}\mathrm{id}$,
$C_{\mathfrak{h},\mathfrak{p}}=h_{\mathfrak{p}}\mathrm{id}$,
$C_{\mathfrak{k},\mathfrak{u}}=k_{\mathfrak{u}}\mathrm{id}$,
$C_{\mathfrak{k},\mathfrak{p}}=k_{\mathfrak{p}}\mathrm{id}$,
$C_{\mathfrak{l},\mathfrak{p}}=l_{\mathfrak{p}}\mathrm{id}$,
where $\mathrm{id}$ denotes the identity transformation.
\end{enumerate}
A basic quadruple $(G,L,K,H)$ is called standard if the standard metrics on $G/H$, $G/K$ and $G/L$ are Einstein.
\end{defn}

We first prove a simple but   useful lemma.
\begin{lem}\label{lem3.2}
Let $(G,L,K,H)$  be a basic  quadruple. If the standard metrics on $G/K$ and $G/H$ are both Einstein, then the constants $h_{\mathfrak{n}}$, $h_{\mathfrak{u}}$, $h_{\mathfrak{p}}$, $k_{\mathfrak{u}}$, $k_{\mathfrak{p}}$, $l_{\mathfrak{p}}$ are given by
\begin{equation}\label{equ3.4}
   h_{\mathfrak{n}}= h_{\mathfrak{u}}=h_{\mathfrak{p}}=\frac{1}{\dim G/H}\sum_{i}(1-\alpha_{i})\dim H_{i},
\end{equation}
\begin{equation}\label{equ3.5}
   k_{\mathfrak{u}}=k_{\mathfrak{p}}=\frac{1}{\dim G/K}\sum_{i}(1-\beta_{i})\dim K_{i},
\end{equation}
\begin{equation}\label{equ3.6}
  l_{\mathfrak{p}}=\frac{1}{\dim G/L}\sum_{i}(1-\gamma_{i})\dim L_{i},
\end{equation}
where $H_{i}$, $K_{i}$, $L_{i}$ are the simple factors of $H$, $K$ and $L$ respectively, and
$B_{\mathfrak{h}_{i}}=\alpha_{i}B|_{\mathfrak{h}_{i}}$, $B_{\mathfrak{k}_{i}}=\beta_{i}B|_{\mathfrak{k}_{i}}$,
$B_{\mathfrak{l}_{i}}=\gamma_{i}B|_{\mathfrak{l}_{i}}$.

Moreover, if there exists constants $c_{1}, c_{2}, c_{3}\in$ such that $B_{\mathfrak{l}}=c_{1}B|_{\mathfrak{l}}$,
$B_{\mathfrak{k}}=c_{2}B|_{\mathfrak{k}}$, and $B_{\mathfrak{h}}=c_{3}B|_{\mathfrak{h}}$, then we have
\begin{equation}\label{equ3.7}
  k_{\mathfrak{u}}=k_{\mathfrak{p}}=\frac{\dim K}{\dim L}l_{\mathfrak{p}},
\end{equation}
\begin{equation}\label{equ3.8}
  h_{\mathfrak{n}}=h_{\mathfrak{u}}=h_{\mathfrak{p}}=\frac{\dim H}{\dim L}l_{\mathfrak{p}},
\end{equation}
\begin{equation}\label{equ3.9}
  c_{1}=1-\frac{\dim G-\dim L}{\dim L}l_{\mathfrak{p}},
\end{equation}
\begin{equation}\label{equ3.10}
  c_{2}=1-\frac{\dim G-\dim K}{\dim L}l_{\mathfrak{p}}.
\end{equation}
\end{lem}
\begin{proof} First,
 \eqref{equ3.4}, \eqref{equ3.5} and \eqref{equ3.6} can be easily calculated by taking the trace of
$C_{\mathfrak{h},\mathfrak{n}}$, $C_{\mathfrak{k},\mathfrak{u}}$ and $C_{\mathfrak{l},\mathfrak{p}}$.
Then \eqref{equ3.7} and \eqref{equ3.8} follows from the facts that
\begin{equation*}
 k_{\mathfrak{p}}=\frac{\dim K}{\dim G/L}(1-c_{1}), \quad h_{\mathfrak{p}}=\frac{\dim H}{\dim G/L}(1-c_{1}).
\end{equation*}
Finally, \eqref{equ3.9}, \eqref{equ3.10} follows from \eqref{equ3.6} and \eqref{equ3.5}.
\end{proof}
 Note that for a general compact simple subgroup $H\subset G$, there always exists a constant $c$ such that $B_{\mathfrak{h}}=cB|_{\mathfrak{h}}$.  The method of computing $c$ is given in \cite{dz79}. In particular, if
 $\mathfrak{h}$ is a regular subalgebra of $\mathfrak{g}$ (see \cite{dynkin57}), then the constant $c$ is given by
 \begin{equation}\label{}
 c=\frac{B_{\mathfrak{h}}(\alpha_{m}',\alpha_{m}')}{ B(\alpha_{m},\alpha_{m})},
\end{equation}
where $\alpha_{m}'$, $\alpha_{m}$ are the maximal root of $\mathfrak{h}$ and $\mathfrak{g}$, respectively.

We must mention that, in this paper, most  subalgebras are regular.
Note also that  the values of $B(\alpha_{m},\alpha_{m})$  for compact simple Lie groups have been  given in Table 3 of  \cite{dz79}. For convenience, we  summarize some of the results as the following table.

\pagebreak
\begin{center}
Table C
\end{center}
\begin{center}
\begin{tabular}{|c|c|c|}
  \hline
  $\mathfrak{g}$ & $B(\alpha_{m},\alpha_{m})$ & dim $\mathfrak{g}$  \\
\hline
  $\mathfrak{su}(n)$ & $4n$  & $n^{2}-1$ \\
   $\mathfrak{so}(n)(n\geq 4)$& $4(n-2)$ & $\frac{1}{2}n(n-1)$ \\
   $\mathfrak{sp}(n)$& $4(n+1)$ & $2n^{2}+n$ \\
   $\mathfrak{g}_{2}$& 16 & 14 \\
  $\mathfrak{f}_{4}$ & 36 & 52 \\
  $\mathfrak{e}_{6}$ & 48 & 78 \\
  $\mathfrak{e}_{7}$ & 72 & 133 \\
  $\mathfrak{e}_{8}$ & 120 & 248 \\
  \hline
\end{tabular}
\end{center}

In the case that $H$ is semisimple,  the following results will be  useful.

\begin{prop}[\cite{dz79}]\label{c1}
Let $G/H$ be a strongly isotropy irreducible space with $H$  not simple. If there exists a constant $c$ such that
$B_{\mathfrak{h}}=cB|_{\mathfrak{h}}$, then the Lie algebra pair $(\mathfrak{g}, \mathfrak{h})$
must be one of the following six cases:
\begin{gather*}
  \mathfrak{so}(k)\oplus \mathfrak{so}(k)\subset \mathfrak{so}(2k), \quad
  \mathfrak{sp}(k)\oplus \mathfrak{sp}(k)\subset \mathfrak{sp}(2k), \\
  \mathfrak{su}(k)\oplus \mathfrak{su}(k)\subset \mathfrak{su}(k^{2}), \quad
  \mathfrak{su}(3)\oplus \mathfrak{su}(3) \oplus \mathfrak{su}(3)\subset \mathfrak{e}_{6},\\
  \mathfrak{sp}(3)\oplus \mathfrak{g}_{2}\subset \mathfrak{e}_{7},\quad
\mathfrak{sp}(1)\oplus \mathfrak{so}(4)\subset \mathfrak{sp}(4).
\end{gather*}
\end{prop}
\begin{thm}[\cite{wz85}]\label{c2}
Let $G$ be a compact connected  simple Lie group and $H$ a semi-simple subgroup such that $G/H$
is standard homogeneous Einstein but not strongly isotropy irreducible. Then
there exists a constant $c$ such that $B_{\mathfrak{h}}=cB|_{\mathfrak{h}}$ except for
the following two cases:
$$\mathfrak{sp}(1)\oplus \mathfrak{sp}(5)\oplus \mathfrak{so}(6)\subset \mathfrak{so}(26),\quad
\mathfrak{so}(8)\oplus 3\mathfrak{su}(2)\subset \mathfrak{e}_{7}.$$
\end{thm}

\textbf{Proof of Theorem \ref{main1}}
\quad Let $(G,L,K,H)$ be a basic quadruple with $G$ compact simple, such that the standard metrics on $G/L$, $G/K$, $G/H$ are Einstein. Then $K/H \rightarrow G/H \rightarrow G/K$ is one of the fibrations listed in Table XI of \cite{wz85}.
Combining Table IA and Table XI of \cite{wz85}, we can find out all the  subgroups $L$ of $G$ which contains  $K $ such that the standard metric on $G/L$ is Einstein. Applying Proposition \ref{c1}
and Theorem \ref{c2}, we can determine all the ones such that there exists a constant $c$ with
$B_{\bar{\mathfrak{l}}}=cB|_{\bar{\mathfrak{l}}}$   among the above subgroups . The result is listed  in  Table A.

On the other hand,  if $H$=\{e\}, then $L/K \rightarrow G/K \rightarrow G/L$ is also a fibration of Einstein metrics  listed in Table XI of \cite{wz85}. According to Definition \ref{def3.1}, we only need to find out all the  subgroups $L$ such that  $K\subset L$ and there exists  a constant
$c_1>0$ with  $B_{\mathfrak{l}}=c_{1}B|_{\mathfrak{l}}$. Combining this with Proposition \ref{c1} and Theorem \ref{c2}, we get  Table B. This completes the proof of Theorem \ref{main1}.

\medskip
There are two types of the basic quadruples which need some more interpretation, namely,
\begin{equation*}
 \textbf{Type A. 3}  \quad \mathfrak{so}(n_{1}n_{2}k)\supset n_{1}\mathfrak{so}(n_{2}k)
  \supset  n_{1}n_{2}\mathfrak{so}(k)\supset \oplus_{i=1}^{l}\mathfrak{h}_{i},  \quad k\geq3, n_{i}\geq2,
\end{equation*}
and
\begin{equation*}
  \textbf{Type B. 2} \quad \mathfrak{so}(nk)\supset n\mathfrak{so}(k)
  \supset \oplus_{i=1}^{l}\mathfrak{h}_{i}\supset \{e\}, \quad k\geq3, n\geq2.
\end{equation*}
These two types of basic quadruples are constructed through the following observation.

Let $G_{i}/H_{i}, i=1,\ldots,l$ $(l\geq 2)$ be a family of irreducible symmetric spaces such that either $H_{i}$ is simple or
$G_{i}/H_{i}$ is one   of the types
$\mathrm{SO}(2k)/\mathrm{SO}(k)\times \mathrm{SO}(k)$ and $\mathrm{Sp}(2k)/\mathrm{Sp}(k)\times \mathrm{Sp}(k)$.
Then $G/H=G_{1}/H_{1}\times \cdots \times G_{s}/H_{s}$ is also a  symmetric space.
Let $\pi$  be the isotropy representation of $G/H$. Then it has been shown in \cite{wz85} that
$\mathrm{SO}(\dim G/H)/\pi(H)$ is a standard homogeneous Einstein manifold if and only if
$\frac{\dim G_{i}}{\dim H_{i}}$ is independent
of $i$.  In particular, if the standard metric on $\mathrm{SO}(\dim G/H)/\pi(H)$ is Einstein, then by Theorem \ref{c2}, there exists a constant $c$ such that
$B_{\mathfrak{h}}=cB|_{\mathfrak{h}}$.
Now in the above two types,
$\oplus_{i=1}^{l}\mathfrak{h}_{i} \subset n\mathfrak{so}(k)\subset \mathfrak{so}(nk)$ can  be expressed
as
\begin{equation}\label{}
  \bigoplus_{i}^{l}\mathfrak{h}_{i}=\bigoplus_{s=0}^{n-1}\bigoplus_{i=t_{s}+1}^{t_{s+1}}\mathfrak{h}_{i}
  \subset n\mathfrak{so}(k), \quad  0=t_{0}<t_{1}<\cdots <t_{n}=l,
\end{equation}
where  $\oplus_{i=t_{s}+1}^{t_{s+1}}\mathfrak{h}_{i}\subset \mathfrak{so}(k)$ and
 $\mathrm{SO}(nk)/\oplus_{i}^{l}H_{i}$,
$\mathrm{SO}(k)/\oplus_{i=t_{s}+1}^{t_{s+1}}H_{i}$ $(s=0, 1, \ldots, n-1)$ are standard homogeneous Einstein manifolds.
Moreover, it is easy to check that
   \begin{equation}\label{}
   \dim H_{i}\leq \dim \mathrm{SO}(\dim G_{i}/H_{i}), \forall 1\leq i\leq l.
   \end{equation}
  Then it follows that
\begin{equation}\label{}
  \dim \oplus_{i=1}^{l}\mathfrak{h}_{i}<\frac{1}{2}\dim n\mathfrak{so}(k).
\end{equation}
This assertion will be useful in the following sections.

\section{Ricci curvature of the invariant metrics}
As  in Section 1, given a   basic  quadruple $(G,L,K,H)$,
we consider $G$-invariant metrics of the form
\begin{equation*}
  \langle , \rangle=g_{(x,y)}=B|_{\mathfrak{n}}+xB|_{\mathfrak{u}}+yB|_{\mathfrak{p}}, \quad x,y\in \mathds{R}^{+},
\end{equation*}
on the homogeneous space $G/H$.
In this section, we mainly study the condition for  $g_{(x,y)}$ to be Einstein.

First, we have
\begin{lem}
Let $(G,L,K,\{e\})$ be a basic quadruple with $G$ simple, then
the left invariant metric $g_{(x,y)}$ on $G$ is naturally reductive with respect to $G\times N$ for some closed subgroup $N$ of $G$, if and only if at least one of the following holds:

 \rm(1) \,\, $x=y$.

 \rm(2) \,\, $x=1$.

 \rm(3) \,\, $\mathfrak{k}$ is an ideal in $\mathfrak{l}$.
\end{lem}
\begin{proof}
 It follows from Theorem \ref{nat-thm}  and the fact that  $\mathfrak{k}$ and $\mathfrak{l}$ are subalgebras of $\mathfrak{g}$.
\end{proof}
The following result is obvious, so we omit the proof.
\begin{prop}
Let $(G,L,K,H)$ and $(G',L',K',H')$ be two basic quadruples, and $g_{(x,y)}$, $g'_{(x',y')}$ be  two invariant metrics on $G/H$ and $G'/H'$ defined as above, respectively.
Then $(G/H,g_{(x,y)})$ is isometric to $(G'/H',g'_{(x',y')})$ if and only if there exists an isomorphism $\varphi: G\rightarrow G'$, such that
$\varphi(L)=L'$,
$\varphi(K)=K'$, $\varphi(H)=H'$, and $x=x'$, $y=y'$.
\end{prop}
Now we compute the Ricci curvature of $g_{(x,y)}$. It is well known that the sectional curvature and Ricci curvature of a homogeneous Riemannian manifold can be explicitly expressed using the inner product on the tangent space and the Lie algebraic structure. In the literature, there are several versions of the formulas. Here we will use the formula of  the Ricci curvature of an invariant metric on a homogeneous compact Riemannian manifold  given by \cite{besse87book} (see (7.38) of \cite{besse87book}):
\begin{equation}\label{7.38}
  \mathrm{Ric}(X,Y)=\frac{1}{2}B(X,Y)-\frac{1}{2}\sum_{i}\langle [X,X_{i}]_{\mathfrak{m}},[Y,X_{i}]_{\mathfrak{m}}\rangle +\frac{1}{4}\sum_{i,j}\langle [X_{i},X_{j}]_{\mathfrak{m}},X\rangle \langle [X_{i},X_{j}]_{\mathfrak{m}},Y\rangle,
\end{equation}
 where $\{X_{i}\}$ is an orthonormal basis of $\mathfrak{m}$ with respect to the restriction of the inner product $\langle,\rangle$ to $\mathfrak{m}$.

 Now we have
\begin{lem}\label{Ric}
Let $(G,L,K,H)$  be a basic  quadruple. Then the Ricci curvature of  $(G/H, g_{(x,y)})$ is given as follows:

\rm(1)\quad $\mathrm{Ric}(n,u)=\mathrm{Ric}(n,p)=\mathrm{Ric}(u,p)=0$,

\rm(2)\quad $\mathrm{Ric}(n,n)=
\big[\frac{1}{4}c_{2}+\frac{1}{2}h_{\mathfrak{n}}+\frac{1}{4x^{2}}(c_{1}-c_{2})+\frac{1}{4y^{2}}(1-c_{1})\big]B(n,n)$,

\rm(3)\quad
$\mathrm{Ric}(u,u)=
\big[\frac{1}{2}k_{\mathfrak{u}}+\frac{1}{4}c_{1}-\frac{1}{2x}(k_{\mathfrak{u}}-h_{\mathfrak{u}})+\frac{x^{2}}{4y^{2}}(1-c_{1})\big]B(u,u)$,

\rm(4)\quad
$\mathrm{Ric}(p,p)=
\big[\frac{1}{4}+\frac{1}{2}l_{\mathfrak{p}}-\frac{1}{2y}(k_{\mathfrak{p}}-h_{\mathfrak{p}})-\frac{x}{2y}(l_{\mathfrak{p}}-k_{\mathfrak{p}})\big]B(p,p)$,

where $ n\in \mathfrak{n}, u\in \mathfrak{u}, p\in \mathfrak{p}$.
\end{lem}

\begin{proof}
The formulas will be proved through   a direct computation.
Let $\{h_{i}\}\subset \mathfrak{h}$, $\{n_{i}\}\subset \mathfrak{n}$, $\{u_{i}\}\subset \mathfrak{u}$, $\{p_{i}\}\subset \mathfrak{p}$
be a B-orthonormal basis of $\mathfrak{g}$. Then
$\{n_{i}\}\cup \{\frac{u_{i}}{\sqrt{x}}\}\cup \{\frac{p_{i}}{\sqrt{y}}\}$ is an orthonormal basis of $\mathfrak{m}$ with respect
to $g_{(x,y)}$. Given   $ n\in \mathfrak{n}$, $u\in \mathfrak{u}$, and $p\in \mathfrak{p}$,  by \eqref{7.38}, one has
\begin{eqnarray*}
  \mathrm{Ric}(u,p) &=& \frac{1}{2}B(u,p)-\frac{1}{2}\sum_{i}\langle [u,n_{i}]_{\mathfrak{m}},[p,n_{i}]_{\mathfrak{m}}\rangle +\frac{1}{4}\sum_{i,j}\langle u,[\frac{p_{i}}{\sqrt{y}},\frac{p_{j}}{\sqrt{y}}]_{\mathfrak{m}}\rangle
   \langle u,[\frac{p_{i}}{\sqrt{y}},\frac{p_{j}}{\sqrt{y}}]_{\mathfrak{m}}\rangle\\
   && -\frac{1}{2}\sum_{i}\langle[u,\frac{u_{i}}{\sqrt{x}}]_{\mathfrak{m}},[p,\frac{u_{i}}{\sqrt{x}}]_{\mathfrak{m}}\rangle -\frac{1}{2}\sum_{i}\langle[u,\frac{p_{i}}{\sqrt{y}}]_{\mathfrak{m}},[p,\frac{p_{i}}{\sqrt{y}}]_{\mathfrak{m}}\rangle\\
   &=& -\frac{1}{2}y\sum_{i}B([u,\frac{p_{i}}{\sqrt{y}}],[p,\frac{p_{i}}{\sqrt{y}}]) +\frac{1}{4}xy\sum_{i,j}B( u,[\frac{p_{i}}{\sqrt{y}},\frac{p_{j}}{\sqrt{y}}])  B(p,[\frac{p_{i}}{\sqrt{y}},\frac{p_{j}}{\sqrt{y}}])\\
  &=&0,
\end{eqnarray*}
Similarly,  $\mathrm{Ric}(n,p)=0$. On the other hand, we have
\begin{eqnarray*}
  \mathrm{Ric}(n,u) &=& \frac{1}{2}B(n,u)-\frac{1}{2}\sum_{i}\langle[n,n_{i}]_{\mathfrak{m}},
  [u,n_{i}]_{\mathfrak{m}}\rangle  +\frac{1}{4}\sum_{i,j}\langle n,[\frac{p_{i}}{\sqrt{y}},\frac{p_{j}}{\sqrt{y}}]_{\mathfrak{m}}\rangle\langle u,[\frac{p_{i}}{\sqrt{y}},\frac{p_{j}}{\sqrt{y}}]_{\mathfrak{m}}\rangle \\
   && -\frac{1}{2}\sum_{i}\langle[n,\frac{u_{i}}{\sqrt{x}}]_{\mathfrak{m}},
   [u,\frac{u_{i}}{\sqrt{x}}]_{\mathfrak{m}}\rangle
    -\frac{1}{2}\sum_{i}\langle[n,\frac{p_{i}}{\sqrt{y}}]_{\mathfrak{m}},
    [u,\frac{p_{i}}{\sqrt{y}}]_{\mathfrak{m}}\rangle\\
   && +\frac{1}{4}\sum_{i,j}\langle n,[n_{i},n_{j}]_{\mathfrak{m}}\rangle\langle u,[n_{i},n_{j}]_{\mathfrak{m}}\rangle
    +\frac{1}{4}\sum_{i,j}\langle n,[\frac{u_{i}}{\sqrt{x}},\frac{u_{j}}{\sqrt{x}}]_{\mathfrak{m}}\rangle\langle u,[\frac{u_{i}}{\sqrt{x}},\frac{u_{j}}{\sqrt{x}}]_{\mathfrak{m}}\rangle\\
   &=&  -\frac{1}{2}\sum_{i}B([n,u_{i}],[u,u_{i}])-\frac{1}{2}\sum_{i}B([n,p_{i}],[u,p_{i}])\\
   && +\frac{1}{4x}\sum_{i,j}B(n,[u_{i},u_{j}])B(u,[u_{i},u_{j}])+\frac{x}{4y^{2}}\sum_{i,j}B(n,[p_{i},p_{j}])B(u,[p_{i},p_{j}])  \\
   &=& \frac{1}{4x}\sum_{i,j}B([n,u_{i}],u_{j})B([u,u_{i}],u_{j})+\frac{x}{4y^{2}}\sum_{i,j}B([n,p_{i}],p_{j})B([u,p_{i}],p_{j})  \\
   &=& \frac{1}{4x}\sum_{i}B([n,u_{i}],[u,u_{i}])+\frac{x}{4y^{2}}\sum_{i}B([n,p_{i}],[u,p_{i}]) \\
   &=&  \frac{1}{4x}B_{\bar{\mathfrak{l}}}(n,u)+\frac{x}{4y^{2}}[B(n,u)-B_{\bar{\mathfrak{l}}}(n,u)]\\
   &=&  0,
\end{eqnarray*}
which proves the first assertion.

Now, a direct calculation shows that
\begin{eqnarray*}
  \mathrm{Ric}(n,n) &=& \frac{1}{2}B(n,n)-\frac{1}{2}\sum_{i}\langle[n,n_{i}]_{\mathfrak{m}},[n,n_{i}]_{\mathfrak{m}}\rangle -\frac{1}{2}\sum_{i}\langle[n,\frac{u_{i}}{\sqrt{x}}]_{\mathfrak{m}},[n,\frac{u_{i}}{\sqrt{x}}]_{\mathfrak{m}}\rangle\\
   && -\frac{1}{2}\sum_{i}\langle[n,\frac{p_{i}}{\sqrt{y}}]_{\mathfrak{m}},[n,\frac{p_{i}}{\sqrt{y}}]_{\mathfrak{m}}\rangle
   +\frac{1}{4}\sum_{i,j}\langle n,[n_{i},n_{j}]_{\mathfrak{m}}\rangle^{2}\\
   &&  +\frac{1}{4}\sum_{i,j}\langle n,[\frac{u_{i}}{\sqrt{x}},\frac{u_{j}}{\sqrt{x}}]_{\mathfrak{m}}\rangle^{2}
   +\frac{1}{4}\sum_{i,j}\langle n,[\frac{p_{i}}{\sqrt{y}},\frac{p_{j}}{\sqrt{y}}]_{\mathfrak{m}}\rangle^{2}\\
   &=& \sum_{i,j}B([n,n_{i}],h_{j})^{2}+\frac{1}{4}\sum_{i,j}B(n,[n_{i},n_{j}])^{2}\\
   &&+\frac{1}{4x^{2}}\sum_{i}B([n,u_{i}],[n,u_{i}])+\frac{1}{4y^{2}}\sum_{i}B([n,p_{i}],[n,p_{i}]).
\end{eqnarray*}
Since
\begin{eqnarray*}
   \sum_{i,j}B(n,[n_{i},n_{j}])^{2} &=& \sum_{i}B([n,n_{i}],[n,n_{i}])- \sum_{i,j}B(n_{j},[n,h_{i}])^{2}\\
   &=&  B_{\bar{\mathfrak{k}}}(n,n)-2B(C_{\mathfrak{h},\mathfrak{n}}(n),n),
\end{eqnarray*}
we have
\begin{eqnarray}
  \mathrm{Ric}(n,n) &=& \sum_{i}B([n,h_{i}],[n,h_{i}])+\frac{1}{4}[B_{\bar{\mathfrak{k}}}(n,n)-2B(C_{\mathfrak{h},\mathfrak{n}}(n),n)] \nonumber\\
   &&+\frac{1}{4x^{2}}[B_{\bar{\mathfrak{l}}}(n,n)-B_{\bar{\mathfrak{k}}}(n,n)]
   +\frac{1}{4y^{2}}[B(n,n)-B_{\bar{\mathfrak{l}}}(n,n)] \nonumber\\
   &=& \big[\frac{1}{4}c_{2}+\frac{1}{2}h_{\mathfrak{n}}+\frac{1}{4x^{2}}(c_{1}-c_{2})+\frac{1}{4y^{2}}(1-c_{1})\big]B(n,n).
\end{eqnarray}
Furthermore, using a similar argument, we get
\begin{eqnarray*}
  \mathrm{Ric}(u,u) &=& \frac{1}{2}B(u,u)-\frac{1}{2}\sum_{i}\langle[u,n_{i}]_{\mathfrak{m}},[u,n_{i}]_{\mathfrak{m}}\rangle -\frac{1}{2}\sum_{i}\langle[u,\frac{u_{i}}{\sqrt{x}}]_{\mathfrak{m}},[u,\frac{u_{i}}{\sqrt{x}}]_{\mathfrak{m}}\rangle\\
   && -\frac{1}{2}\sum_{i}\langle[u,\frac{p_{i}}{\sqrt{y}}]_{\mathfrak{m}},[u,\frac{p_{i}}{\sqrt{y}}]_{\mathfrak{m}}\rangle
   +\frac{1}{2}\sum_{i,j}\langle u,[n_{i},\frac{u_{j}}{\sqrt{x}}]_{\mathfrak{m}}\rangle^{2}\\
   &&  +\frac{1}{4}\sum_{i,j}\langle u,[\frac{u_{i}}{\sqrt{x}},\frac{u_{j}}{\sqrt{x}}]_{\mathfrak{m}}\rangle^{2}
   +\frac{1}{4}\sum_{i,j}\langle u,[\frac{p_{i}}{\sqrt{y}},\frac{p_{j}}{\sqrt{y}}]_{\mathfrak{m}}\rangle^{2}\\
   &=&  \frac{1}{2}B(u,u)-\frac{x}{2}\sum_{i}B([u,n_{i}],[u,n_{i}]) -\frac{1}{2x}\sum_{i}\langle[u,u_{i}]_{\mathfrak{m}},[u,u_{i}]_{\mathfrak{m}}\rangle \\
   && -\frac{1}{2}\sum_{i}B([u,p_{i}],[u,p_{i}])+\frac{x}{2}\sum_{i,j}B(u,[n_{i},u_{j}])^{2} \\
   && +\frac{1}{4}\sum_{i,j}B(u,[u_{i},u_{j}])^{2}+\frac{x^{2}}{4y^{2}}\sum_{i,j}B(u,[p_{i},p_{j}])^{2}.
\end{eqnarray*}
Next, since
\begin{eqnarray*}
   &&  \sum_{i,j}B(u,[u_{i},u_{j}])^{2}\\
   &=& \sum_{i,j}B([u,u_{i}],u_{j})^{2} \\
   &=& \sum_{i}B([u,u_{i}],[u,u_{i}])-\sum_{i,j}B([u,u_{i}],n_{j})^{2}-\sum_{i,j}B([u,u_{i}],h_{j})^{2} \\
   &=& B_{\bar{\mathfrak{l}}}(u,u)-2\sum_{i}B([u,n_{i}],[u,n_{i}])-2\sum_{i}B([u,h_{i}],[u,h_{i}]) \\
   &=& (c_{1}-2k_{\mathfrak{u}})B(u,u),
\end{eqnarray*}
we have
\begin{eqnarray*}
   &&  \sum_{i}\langle[u,u_{i}]_{\mathfrak{m}},[u,u_{i}]_{\mathfrak{m}}\rangle\\
   &=&  \sum_{i,j}\big[\langle[u,u_{i}]_{\mathfrak{m}},\frac{u_{j}}{\sqrt{x}}\rangle^{2}+\langle[u,u_{i}]_{\mathfrak{m}},n_{j}\rangle^{2}\big]\\
   &=&  \sum_{i,j}\big[xB([u,u_{i}],u_{j})^{2}+B([u,u_{i}],n_{j})^{2}\big]\\
   &=& \sum_{i,j}\big[xB(u,[u_{i},u_{j}])^{2}\big]+\sum_{i}B([u,n_{i}],[u,n_{i}]) \\
   &=& (c_{1}-2k_{\mathfrak{u}})xB(u,u)+(k_{\mathfrak{u}}-h_{\mathfrak{u}})B(u,u).
\end{eqnarray*}
Therefore we have
\begin{eqnarray}
  \mathrm{Ric}(u,u) &=& \frac{1}{2}B(u,u)-\frac{1}{2x}\big[(c_{1}-2k_{\mathfrak{u}})xB(u,u)+(k_{\mathfrak{u}}-h_{\mathfrak{u}})B(u,u)\big] \nonumber\\
   && -\frac{1}{2}(1-c_{1})B(u,u)+\frac{1}{4}(c_{1}-2k_{\mathfrak{u}})B(u,u) \nonumber\\
   &&  +\frac{x^{2}}{4y^{2}}(1-c_{1})B(u,u)\nonumber\\
   &=&  \big[\frac{1}{2}k_{\mathfrak{u}}+\frac{1}{4}c_{1}-
   \frac{1}{2x}(k_{\mathfrak{u}}-h_{\mathfrak{u}})+\frac{x^{2}}{4y^{2}}(1-c_{1})\big]B(u,u),
\end{eqnarray}
and
\begin{eqnarray*}
  \mathrm{Ric}(p,p) &=& \frac{1}{2}B(p,p)-\frac{1}{2}\sum_{i}\langle[p,n_{i}]_{\mathfrak{m}},[p,n_{i}]_{\mathfrak{m}}\rangle \\
   && -\frac{1}{2}\sum_{i}\langle[p,\frac{u_{i}}{\sqrt{x}}]_{\mathfrak{m}},[p,\frac{u_{i}}{\sqrt{x}}]_{\mathfrak{m}}\rangle -\frac{1}{2}\sum_{i}\langle[p,\frac{p_{i}}{\sqrt{y}}]_{\mathfrak{m}},[p,\frac{p_{i}}{\sqrt{y}}]_{\mathfrak{m}}\rangle \\
   && +\frac{1}{2}\sum_{i,j}\langle p,[n_{i},\frac{p_{j}}{\sqrt{y}}]_{\mathfrak{m}}\rangle^{2} +\frac{1}{2}\sum_{i,j}\langle p,[\frac{u_{i}}{\sqrt{x}},\frac{p_{j}}{\sqrt{y}}]_{\mathfrak{m}}\rangle^{2}
   +\frac{1}{4}\sum_{i,j}\langle p,[\frac{p_{i}}{\sqrt{y}},\frac{p_{j}}{\sqrt{y}}]_{\mathfrak{m}}\rangle^{2}\\
   &=& \frac{1}{2}B(p,p)-\frac{y}{2}\sum_{i}B([p,n_{i}],[p,n_{i}]) \\
   && -\frac{y}{2x}\sum_{i}B([p,u_{i}],[p,u_{i}])-\frac{1}{2y}\sum_{i}\langle[p,p_{i}]_{\mathfrak{m}},[p,p_{i}]_{\mathfrak{m}}\rangle \\
   &&+ \frac{y}{2}\sum_{i}B([p,n_{i}],[p,n_{i}])+ \frac{y}{2x}\sum_{i}B([p,u_{i}],[p,u_{i}])
   +\frac{1}{4}\sum_{i,j}B(p,[p_{i},p_{j}])^{2}\\
   &=& \frac{1}{2}B(p,p)-\frac{1}{2y}\sum_{i}\langle[p,p_{i}]_{\mathfrak{m}},[p,p_{i}]_{\mathfrak{m}}\rangle
   +\frac{1}{4}(1-2l_{\mathfrak{p}})B(p,p).
\end{eqnarray*}
Finally, since
\begin{eqnarray*}
   && \sum_{i}\langle[p,p_{i}]_{\mathfrak{m}},[p,p_{i}]_{\mathfrak{m}}\rangle \\
   &=& \sum_{i,j}\big[\langle[p,p_{i}]_{\mathfrak{m}},\frac{p_{j}}{\sqrt{y}}\rangle^{2}+\langle[p,p_{i}]_{\mathfrak{m}},\frac{u_{j}}{\sqrt{x}}\rangle^{2}+\langle[p,p_{i}]_{\mathfrak{m}},n_{j}\rangle^{2}\big] \\
   &=&  \sum_{i,j}\big[yB([p,p_{i}],p_{j})^{2}+x^{2}B([p,p_{i}],\frac{u_{j}}{\sqrt{x}})^{2}+B([p,p_{i}],n_{j})^{2}\big]\\
   &=&  (1-2l_{\mathfrak{p}})yB(p,p)+x\sum_{i}B([p,u_{i}],[p,u_{i}])+\sum_{i}B([p,n_{i}],[p,n_{i}])\\
   &=&  (1-2l_{\mathfrak{p}})yB(p,p)+x(l_{\mathfrak{p}}-k_{\mathfrak{p}})B(p,p)+(k_{\mathfrak{p}}-h_{\mathfrak{p}})B(p,p),
\end{eqnarray*}
we have
\begin{eqnarray}
  \mathrm{Ric}(p,p) &=& \frac{1}{2}B(p,p)- \frac{1}{2y}\big[(1-2l_{\mathfrak{p}})y+x(l_{\mathfrak{p}}-k_{\mathfrak{p}})+(k_{\mathfrak{p}}-h_{\mathfrak{p}})\big]B(p,p)  \nonumber\\
  &&+\frac{1}{4}(1-2l_{\mathfrak{p}})B(p,p)\nonumber\\
   &=& \big[\frac{1}{4}+\frac{1}{2}l_{\mathfrak{p}}-\frac{1}{2y}(k_{\mathfrak{p}}-h_{\mathfrak{p}})-\frac{x}{2y}(l_{\mathfrak{p}}-k_{\mathfrak{p}})\big]B(p,p).
\end{eqnarray}
This completes the proof of the lemma.
\end{proof}

\begin{prop}\label{prop4.4}
Let $(G,L,K,H)$  be a basic  quadruple. Then the following two assertions hold:
\begin{enumerate}
 \item If $h_{\mathfrak{n}}\neq h_{\mathfrak{u}}$, then the invariant metric $g_{(x,y)}$ on $G/H$ is Einstein if and only if
$(x,y)$ satisfies the following equations:
{\small \begin{eqnarray}\label{Einstein equ}
  &&\frac{1-c_{1}}{4}(\frac{1}{4}+\frac{1}{2}l_{\mathfrak{p}})^{2}x^{2}(x-1)\Delta(x)= \nonumber\\
 &&  \Big(\big[\frac{1}{2}(k_{\mathfrak{p}}-h_{\mathfrak{p}})+\frac{1-c_{1}}{4}
  +\frac{x}{2}(l_{\mathfrak{p}}-k_{\mathfrak{p}})\big]\Delta(x)
     +\frac{1-c_{1}}{4}(x-1)
  \big[(\frac{1}{4}c_{2}+\frac{1}{2}h_{\mathfrak{n}})x^{2}+\frac{c_{1}-c_{2}}{4}\big]\Big)^{2},
\end{eqnarray}}
\begin{equation}\label{x=y}
  y=x\sqrt{\frac{(1-c_{1})(x-1)}{4\Delta(x)}},
\end{equation}
where
\begin{equation}\label{}
  \Delta(x)=(\frac{1}{4}c_{2}+\frac{1}{2}h_{\mathfrak{n}})x^{2}
  -(\frac{1}{2}k_{\mathfrak{u}}+\frac{1}{4}c_{1})x+\frac{1}{2}(k_{\mathfrak{u}}-h_{\mathfrak{u}})+\frac{c_{1}-c_{2}}{4}.
\end{equation}
\item If $h_{\mathfrak{n}}=h_{\mathfrak{u}}$, then invariant metric $g_{(1,y)}$ on $G/H$ is Einstein if and only if
 $y$ satisfies the following equation:
 \begin{equation}\label{}
   (\frac{c_{1}}{4}+\frac{1}{2}h_{\mathfrak{n}})y^{2}
   -(\frac{1}{4}+\frac{1}{2}l_{\mathfrak{p}})y
   +\frac{1}{2}(\frac{1}{2}+l_{\mathfrak{p}}-\frac{c_{1}}{2}-h_{\mathfrak{p}})=0.
 \end{equation}
 Moreover,  in this case, the invariant metric $g_{(x,y)}\, (x\neq1)$ on $G/H$ is Einstein if and only if
$(x,y)$ satisfies the conditions:
\begin{eqnarray}\label{equ4.18}
  &&\frac{1-c_{1}}{4}(\frac{1}{4}+\frac{1}{2}l_{\mathfrak{p}})^{2}x^{2}\delta(x)
  =\nonumber\\
  &&\Big(\big[\frac{1}{2}(k_{\mathfrak{p}}-h_{\mathfrak{p}})
  +\frac{1-c_{1}}{4}+\frac{x}{2}(l_{\mathfrak{p}}-k_{\mathfrak{p}})\big]\delta(x)+\frac{1-c_{1}}{4}
  \big[(\frac{1}{4}c_{2}+\frac{1}{2}h_{\mathfrak{n}})x^{2}+\frac{c_{1}-c_{2}}{4}\big]\Big)^{2},
\end{eqnarray}
and
\begin{equation}\label{}
  y=x\sqrt{\frac{1-c_{1}}{4\delta(x)}},
\end{equation}
where
\begin{equation}\label{}
  \delta(x)=(\frac{1}{4}c_{2}+\frac{1}{2}h_{\mathfrak{n}})x
  -\frac{1}{2}(k_{\mathfrak{u}}-h_{\mathfrak{u}})-\frac{c_{1}-c_{2}}{4}.
\end{equation}
\end{enumerate}
\end{prop}
\begin{proof}
By Lemma \ref{Ric}, the invariant metric $g_{(x,y)}$ on $G/H$ is Einstein with Ricci constant $\lambda$ if and only if
$(x,y)$ satisfies the following equations:
\begin{equation}\label{equ1}
 \frac{1}{4}c_{2}+\frac{1}{2}h_{\mathfrak{n}}+\frac{1}{4x^{2}}(c_{1}-c_{2})+\frac{1}{4y^{2}}(1-c_{1})=\lambda,
\end{equation}
\begin{equation}\label{equ2}
  \frac{1}{2}k_{\mathfrak{u}}+\frac{1}{4}c_{1}-\frac{1}{2x}(k_{\mathfrak{u}}-h_{\mathfrak{u}})
  +\frac{x^{2}}{4y^{2}}(1-c_{1})=\lambda x,
\end{equation}
\begin{equation}\label{equ3}
  \frac{1}{4}+\frac{1}{2}l_{\mathfrak{p}}-\frac{1}{2y}(k_{\mathfrak{p}}-h_{\mathfrak{p}})
  -\frac{x}{2y}(l_{\mathfrak{p}}-k_{\mathfrak{p}})=\lambda y.
\end{equation}

Now assume $(x=1,y)$ is a solution of equations \eqref{equ1}, \eqref{equ2} and \eqref{equ3}. Then one has
\begin{equation*}
  \frac{1}{4}c_{2}+\frac{1}{2}h_{\mathfrak{n}}+\frac{1}{4}(c_{1}-c_{2})+\frac{1}{4y^{2}}(1-c_{1})=\lambda
  =\frac{1}{2}k_{\mathfrak{u}}+\frac{1}{4}c_{1}-\frac{1}{2}(k_{\mathfrak{u}}-h_{\mathfrak{u}})
  +\frac{1}{4y^{2}}(1-c_{1}),
\end{equation*}
hence
\begin{equation}\label{}
h_{\mathfrak{n}}=h_{\mathfrak{u}}.
\end{equation}
Moreover, plugging \eqref{equ1} into \eqref{equ3}, we have
\begin{equation}\label{}
  \frac{1}{4}+\frac{1}{2}l_{\mathfrak{p}}-\frac{1}{2y}(k_{\mathfrak{p}}-h_{\mathfrak{p}})
  -\frac{1}{2y}(l_{\mathfrak{p}}-k_{\mathfrak{p}})=[\frac{1}{2}h_{\mathfrak{n}}
  +\frac{1}{4}c_{1}+\frac{1}{4y^{2}}(1-c_{1})]y.
\end{equation}
Thus
\begin{equation}\label{}
   (\frac{c_{1}}{4}+\frac{1}{2}h_{\mathfrak{n}})y^{2}
   -(\frac{1}{4}+\frac{1}{2}l_{\mathfrak{p}})y
   +\frac{1}{2}(\frac{1}{2}+l_{\mathfrak{p}}-\frac{c_{1}}{2}-h_{\mathfrak{p}})=0.
 \end{equation}

Now assume  $(x,y)$ $(x\neq 1)$ is a solution of equations \eqref{equ1}, \eqref{equ2} and \eqref{equ3}.
Then plugging \eqref{equ1} into \eqref{equ2}, we get
\begin{equation*}
  \frac{1}{2}k_{\mathfrak{u}}+\frac{1}{4}c_{1}
  -\frac{1}{2x}(k_{\mathfrak{u}}-h_{\mathfrak{u}})+\frac{x^{2}}{4y^{2}}(1-c_{1})
  =\big[\frac{1}{4}c_{2}+\frac{1}{2}h_{\mathfrak{n}}+\frac{1}{4x^{2}}(c_{1}-c_{2})
  +\frac{1}{4y^{2}}(1-c_{1})\big]x.
\end{equation*}
Therefore we have
\begin{equation}\label{D}
  y^{2}\Delta(x)=\frac{1-c_{1}}{4}x^{2}(x-1),
\end{equation}
and
\begin{equation}\label{}
 y=x\sqrt{\frac{(1-c_{1})(x-1)}{4\Delta(x)}},
\end{equation}
where
\begin{equation}\label{}
  \Delta(x)=(\frac{1}{4}c_{2}+\frac{1}{2}h_{\mathfrak{n}})x^{2}
  -(\frac{1}{2}k_{\mathfrak{u}}+\frac{1}{4}c_{1})x
  +\frac{1}{2}(k_{\mathfrak{u}}-h_{\mathfrak{u}})+\frac{c_{1}-c_{2}}{4}.
\end{equation}

Now plugging \eqref{equ1} into \eqref{equ3}, we have
\begin{equation}\label{D1-1}
  \frac{1}{4}+\frac{1}{2}l_{\mathfrak{p}}-\frac{1}{2y}(k_{\mathfrak{p}}-h_{\mathfrak{p}})
  -\frac{x}{2y}(l_{\mathfrak{p}}-k_{\mathfrak{p}})
  =\big[\frac{1}{4}c_{2}+\frac{1}{2}h_{\mathfrak{n}}+\frac{1}{4x^{2}}(c_{1}-c_{2})+\frac{1}{4y^{2}}(1-c_{1})\big]y.
\end{equation}
Then we have
\begin{equation*}
  (\frac{1}{4}+\frac{1}{2}l_{\mathfrak{p}})y
  -\frac{1}{2}(k_{\mathfrak{p}}-h_{\mathfrak{p}})-\frac{x}{2}(l_{\mathfrak{p}}-k_{\mathfrak{p}})
  =\big[\frac{1}{4}c_{2}+\frac{1}{2}h_{\mathfrak{n}}
  +\frac{1}{4x^{2}}(c_{1}-c_{2})+\frac{1}{4y^{2}}(1-c_{1})\big]y^{2},
\end{equation*}
and
\begin{equation}\label{D1}
 (\frac{1}{4}+\frac{1}{2}l_{\mathfrak{p}})y\Delta(x)=
  \big[(\frac{1}{4}c_{2}+\frac{1}{2}h_{\mathfrak{n}}+\frac{c_{1}-c_{2}}{4x^{2}})y^{2}
  +\frac{1-c_{1}}{4}+\frac{1}{2}(k_{\mathfrak{p}}-h_{\mathfrak{p}})
  +\frac{x}{2}(l_{\mathfrak{p}}-k_{\mathfrak{p}})\big]\Delta(x).
\end{equation}
Now substituting  \eqref{D} into \eqref{D1}, we obtain
\begin{equation}\label{}
\begin{split}
  \frac{1-c_{1}}{4}(\frac{1}{4}+\frac{1}{2}l_{\mathfrak{p}})^{2}x^{2}(x-1)\Delta(x)
  =& \Big(\big[\frac{1}{2}(k_{\mathfrak{p}}-h_{\mathfrak{p}})+\frac{1-c_{1}}{4}
  +\frac{x}{2}(l_{\mathfrak{p}}-k_{\mathfrak{p}})\big]\Delta(x)\\
  &+\frac{1-c_{1}}{4}(x-1)
  \big[(\frac{1}{4}c_{2}+\frac{1}{2}h_{\mathfrak{n}})x^{2}+\frac{c_{1}-c_{2}}{4}\big]\Big)^{2}.
  \end{split}
\end{equation}

Notice that if $h_{\mathfrak{n}}=h_{\mathfrak{u}}$, then $\Delta(x)=(x-1)\delta(x)$, where
\begin{equation}\label{}
  \delta(x)=(\frac{1}{4}c_{2}+\frac{1}{2}h_{\mathfrak{n}})x
  -\frac{1}{2}(k_{\mathfrak{u}}-h_{\mathfrak{u}})-\frac{c_{1}-c_{2}}{4}.
\end{equation}
Thus, in this case,  equation (\ref{Einstein equ}) can be divided by $(x-1)^{2}$, which leads to the following equation:
\begin{eqnarray}\label{}
 && \frac{1-c_{1}}{4}(\frac{1}{4}+\frac{1}{2}l_{\mathfrak{p}})^{2}x^{2}\delta(x)=\nonumber\\
 && \Big(\big[\frac{1}{2}(k_{\mathfrak{p}}-h_{\mathfrak{p}})+\frac{1-c_{1}}{4}
  +\frac{x}{2}(l_{\mathfrak{p}}-k_{\mathfrak{p}})\big]\delta(x)+\frac{1-c_{1}}{4}
  \big[(\frac{1}{4}c_{2}+\frac{1}{2}h_{\mathfrak{n}})x^{2}+\frac{c_{1}-c_{2}}{4}\big]\Big)^{2}.
\end{eqnarray}

Conversely, if $x=z\neq 1$ is a solution of \eqref{Einstein equ}, then the combination of conditions $\Delta(z)\neq 0$ with
 the equation \eqref{D1} is equivalent to the equation \eqref{D1-1},
hence the system of equations \eqref{Einstein equ}, \eqref{x=y} is equivalent to the system of equations \eqref{equ1}, \eqref{equ2}, and \eqref{equ3}.
  This completes the proof of the proposition.
\end{proof}

Notice that the equation \eqref{Einstein equ} is an equation of order six in one variable, hence it might admit no real solutions. Moreover, if the isotropy representation of $H$ on $T_{eH}(G/H)$ decomposes into exactly three non-equivalent irreducible summands, then the $G$-invariant metrics must be of the form $g_{(x,y)}$ up to scaling.
These facts may provide us with a method to obtain new homogeneous spaces which admit no $G$-invariant Einstein metrics.
However, we will not deal with this problem here.

\section{Einstein metrics on normal homogeneous Einstein manifolds}

To prove the main theorem of this paper, we need the following result.
\begin{prop}\label{prop4.5}
Keep the notation as above. Let $(G,L,K,H)$  be a  standard  quadruple listed in Table A and Table B, and denote
$\omega_{1}=\frac{1}{4}+\frac{1}{2}l_{\mathfrak{p}}-k_{\mathfrak{p}}-\frac{c_{1}}{2}$,
$\omega_{2}=2k_{\mathfrak{p}}+c_{1}-2c_{2}-4h_{\mathfrak{p}}$.
Then we have $\omega_{1}\geq 0$, $\omega_{2}\geq 0$ except for the following cases:
\begin{description}
\item{(a)} Type A. 4. $n_{1}=n_{2}=2.$
$$\mathfrak{sp}(4n_{3}k)\supset 2\mathfrak{sp}(2n_{3}k) \supset 4\mathfrak{sp}(n_{3}k)
\supset 4n_{3}\mathfrak{sp}(k), \quad k\geq1, n_{3}\geq2.$$
$$\omega_{1}=-\frac{1}{2(4n_{3}k+1)},\quad \omega_{2}=\frac{2n_{3}k-4k-2}{4n_{3}k+1}.$$
Or $n_{2}=n_{3}=2.$
$$\mathfrak{sp}(4n_{1}k)\supset n_{1}\mathfrak{sp}(4k) \supset 2n_{1}\mathfrak{sp}(2k)
\supset 4n_{1}\mathfrak{sp}(k), \quad k\geq1, n_{1}\geq2.$$
$$\omega_{1}=\frac{2n_{1}k-4k-1}{2(4n_{1}k+1)},\quad \omega_{2}=-\frac{2}{4n_{1}k+1}.$$
\item{(b)} Type A. 5.
$$\mathfrak{e}_{6}\supset \mathfrak{so}(10)\oplus \mathds{R}\supset \mathfrak{so}(8)\oplus \mathds{R}^{2}\supset \mathds{R}^{6}.$$
$$\omega_{1}=-\frac{1}{6}, \quad \omega_{2}=0.$$
\item{(c)} Type A. 6.
 $$\mathfrak{e}_{7}\supset \mathfrak{so}(12)\oplus \mathfrak{su}(2)\supset \mathfrak{so}(8)\oplus 3\mathfrak{su}(2)\supset 7\mathfrak{su}(2).$$
 $$\omega_{1}=-\frac{1}{18}, \quad \omega_{2}=-\frac{2}{9}.$$
\item{(d)} Type B. 3. $n_{1}=n_{2}=2$.
 $$\mathfrak{sp}(4k)\supset 2\mathfrak{sp}(2k)\supset 4\mathfrak{sp}(k)\supset \{e\}, \quad k\geq 1.$$
 $$\omega_{1}=-\frac{1}{2(4k+1)},\quad \omega_{2}=\frac{2k}{4k+1}.$$
\item{(e)} Type B. 4.
 $$\mathfrak{so}(8)\supset \mathfrak{so}(7)\supset \mathfrak{g}_{2}\supset \{e\}.$$
  $$\omega_{1}=-\frac{1}{4},\quad \omega_{2}=\frac{1}{6}.$$
\item{(f)} Type B. 5.
 $$\mathfrak{f}_{4}\supset \mathfrak{so}(9)\supset \mathfrak{so}(8)\supset \{e\}.$$
  $$\omega_{1}=-\frac{5}{18},\quad \omega_{2}=\frac{2}{9}.$$
  \end{description}
\end{prop}
\begin{proof}
Let $(G,L,K,H)$ be one of the standard quadruples listed in Table A and Table B which is not of  Type A 1, A 5, or A 6.  Then there exist
constants $c_{1}$, $c_{2}$, $c_{3}$ such that $B_{\mathfrak{l}}=c_{1}B|_{\mathfrak{l}}$,  $B_{\mathfrak{k}}=c_{2}B|_{\mathfrak{k}}$, and $B_{\mathfrak{h}}=c_{3}B|_{\mathfrak{h}}$.
By Lemma \ref{lem3.2}, one has
\begin{equation}\label{}
  l_{\mathfrak{p}}=\frac{\dim L}{\dim G/L}(1-c_{1}), \quad k_{\mathfrak{p}}=\frac{\dim K}{\dim L}l_{\mathfrak{p}}, \quad
   h_{\mathfrak{p}}=\frac{\dim H}{\dim L}l_{\mathfrak{p}}.
\end{equation}
Therefore we have
\begin{eqnarray}
  \omega_{1} &=& \frac{1}{4}+\frac{1}{2}l_{\mathfrak{p}}-k_{\mathfrak{p}}-\frac{c_{1}}{2}\nonumber \\
   &=& \frac{1}{4}+\frac{1}{2}l_{\mathfrak{p}}-\frac{\dim K}{\dim L}l_{\mathfrak{p}}
   -\frac{1}{2}(1-\frac{\dim G-\dim L}{\dim L}l_{\mathfrak{p}}) \nonumber\\
   &=& -\frac{1}{4}+  \frac{\dim G-2\dim K}{2\dim L}l_{\mathfrak{p}},
\end{eqnarray}
and
\begin{eqnarray}\label{4.41}
  \omega_{2} &=& 2k_{\mathfrak{p}}+c_{1}-2c_{2}-4h_{\mathfrak{p}} \nonumber\\
   &=& 2\frac{\dim K}{\dim L}l_{\mathfrak{p}}+(1-\frac{\dim G-\dim L}{\dim L}l_{\mathfrak{p}})-4\frac{\dim H}{\dim L}l_{\mathfrak{p}}\nonumber\\
   &&-2(1-\frac{\dim G-\dim K}{\dim L}l_{\mathfrak{p}})\nonumber\\
   &=&  \frac{\dim G+\dim L-4\dim H}{\dim L}l_{\mathfrak{p}}-1.
\end{eqnarray}
In particular, if $G/L$ is also a symmetric space, then $l_{\mathfrak{p}}=\frac{1}{2}$,  and hence we have
\begin{eqnarray}
  \omega_{1} &=& \frac{1}{4\dim L}(\dim G-\dim L-2\dim K),\label{omega1} \\
  \omega_{2} &=& \frac{1}{2\dim L}(\dim G-\dim L-4\dim H). \label{omega2}
\end{eqnarray}
Moreover, it is obvious that, if $H=\{e\}$, then $\omega_{2}>0$.

We first consider the cases of Type A. 7, 8, 10, 11 and Type B. 8, 9, 10, 11, 12. In these cases, $G/L=E_{8}/\mathrm{SO}(16)$ is symmetric, and we have $l_{\mathfrak{p}}=\frac{1}{2}$.
Then by \eqref{omega1} and \eqref{omega2}, we have
\begin{eqnarray*}
  \omega_{1} &=& \frac{1}{4\dim L}(\dim G-\dim L-2\dim K) \\
   &\geq& \frac{1}{480}(248-120-2\times 56) \\
   &>& 0,
\end{eqnarray*}
and
\begin{eqnarray*}
  \omega_{2} &=& \frac{1}{2\dim L}(\dim G-\dim L-4\dim H) \\
   &\geq& \frac{1}{240}(248-120-4\times 24) \\
   &>& 0,
\end{eqnarray*}
where we have used the facts that $\dim G=\dim \mathfrak{e}_{8}=248$, $\dim L=\dim \mathfrak{so}(16)=120$,
$\dim K\leq \dim 2\mathfrak{so}(8)=56$, and $\dim H\leq \dim 8\mathfrak{su}(2)=24$.

Next we  will give an explicit description of the related quantities for the rest cases listed in Table A and Table B. Since the computations are somehow repetitive and rather lengthy, we
collect the description in Appendix A.

Now  the proof of the proposition is completed by the above arguments and the description in Appendix A.
\end{proof}

Now we can prove the main theorem of this section.
\begin{thm}\label{main}
Let $(G,L,K,H)$  be a basic  quadruple,
such that $k_{\mathfrak{u}}=k_{\mathfrak{p}}$, $h_{\mathfrak{u}}=h_{\mathfrak{p}}$.
Then $G/H$ admits at least one invariant
Einstein metric of the form $g_{(x,y)}$, with $x\neq 1$, $x\neq y$, if one of the following conditions holds:

(1)\quad $h_{\mathfrak{n}}< h_{\mathfrak{u}};$

(2)\quad $h_{\mathfrak{n}}= h_{\mathfrak{u}}$, and $G$ is simple. That is,  $(G,L,K,H)$ is one of the standard quadruples listed in Table A and Table B which is not the following ones:

(I) Type A. 4. $n_{1}=9m+1$, $n_{2}=n_{3}=2$, $k=2m$, $m\in \mathds{N}^{+}$.
\begin{equation*}
  \mathfrak{sp}(8m(9m+1))\supset (9m+1)\mathfrak{sp}(8m)\supset 2(9m+1)\mathfrak{sp}(4m)\supset 4(9m+1)\mathfrak{sp}(2m).
\end{equation*}

(II) Type A. 5.
 \begin{equation*}
  \mathfrak{e}_{6}\supset \mathfrak{so}(10)\oplus \mathds{R}
  \supset \mathfrak{so}(8)\oplus \mathds{R}^{2}\supset \mathds{R}^{6}.
\end{equation*}

(III) Type B. 3. $n_{1}=n_{2}=2$, $k=1$.
\begin{equation*}
  \mathfrak{sp}(4)\supset 2\mathfrak{sp}(2)\supset 4\mathfrak{sp}(1)\supset \{e\}.
\end{equation*}
\end{thm}
\begin{proof}
Keep the notation as above. Suppose
\begin{equation}\label{M}
  \Delta(x)-\frac{1-c_{1}}{4}(x-1)=M(x-\alpha)(x-\beta),
\end{equation}
where $M=\frac{1}{4}c_{2}+\frac{1}{2}h_{\mathfrak{n}}>0$, and $\alpha,\beta\in \mathds{C}$.
Then it  follows easily from
 (\ref{x=y}) that $x=y$ if and only if $x=\alpha$, or $x=\beta$.

Now plugging (\ref{M}) into the right side of (\ref{Einstein equ}), one has
\begin{eqnarray}
   &&\big[\frac{1}{2}(k_{\mathfrak{p}}-h_{\mathfrak{p}})+\frac{1-c_{1}}{4}
   +\frac{x}{2}(l_{\mathfrak{p}}-k_{\mathfrak{p}})\big]\Delta(x)+\frac{1-c_{1}}{4}(x-1)
  \big[(\frac{1}{4}c_{2}+\frac{1}{2}h_{\mathfrak{\mathfrak{n}}})x^{2}+\frac{c_{1}-c_{2}}{4}\big] \nonumber\\
   &=&  \big[\frac{1}{2}(k_{\mathfrak{p}}-h_{\mathfrak{p}})+\frac{1-c_{1}}{4}
   +\frac{x}{2}(l_{\mathfrak{p}}-k_{\mathfrak{p}})\big]M(x-\alpha)(x-\beta)\nonumber\\
   &&+\frac{1-c_{1}}{4}(x-1)\big[\frac{1}{2}(k_{\mathfrak{p}}-h_{\mathfrak{p}})
   +\frac{1-c_{1}}{4}+\frac{x}{2}(l_{\mathfrak{p}}-k_{\mathfrak{p}})+
   (\frac{1}{4}c_{2}+\frac{1}{2}h_{\mathfrak{n}})x^{2}+\frac{c_{1}-c_{2}}{4}\big]  \nonumber\\
   &=& \big[\frac{1}{2}(k_{\mathfrak{p}}-h_{\mathfrak{p}})+\frac{1-c_{1}}{4}
   +\frac{x}{2}(l_{\mathfrak{p}}-k_{\mathfrak{p}})\big]M(x-\alpha)(x-\beta) \nonumber\\
   &&+ \frac{1-c_{1}}{4}(x-1)\big[M(x-\alpha)(x-\beta)+(\frac{1}{4}+\frac{1}{2}l_{\mathfrak{p}})x\big] \nonumber\\
   &=&  M(x-\alpha)(x-\beta)\eta(x)+\frac{1-c_{1}}{4}(\frac{1}{4}+\frac{1}{2}l_{\mathfrak{p}})(x-1)x,
\end{eqnarray}
where
\begin{equation}
  \eta(x)=\frac{x}{2}(\frac{1}{2}+l_{\mathfrak{p}}-k_{\mathfrak{p}}-\frac{c_{1}}{2})
  +\frac{1}{2}(k_{\mathfrak{p}}-h_{\mathfrak{p}}).
\end{equation}
Then equation (\ref{Einstein equ}) can be simplified as:
\begin{equation*}
  \frac{1-c_{1}}{4}(\frac{1}{4}+\frac{1}{2}l_{\mathfrak{p}})^{2}x^{2}(x-1)\Delta(x)
  = \big[M(x-\alpha)(x-\beta)\eta(x)+\frac{1-c_{1}}{4}(\frac{1}{4}+\frac{1}{2}l_{\mathfrak{p}})(x-1)x\big]^{2}.
\end{equation*}
This implies that
\begin{equation}\label{}
  M(x-\alpha)(x-\beta)\big[M(x-\alpha)(x-\beta)\eta^{2}(x)
  +\frac{1-c_{1}}{4}(\frac{1}{2}+l_{\mathfrak{p}})x(x-1)
  (\eta(x)-\frac{1}{2}(\frac{1}{4}+\frac{1}{2}l_{\mathfrak{p}})x)\big]=0.
\end{equation}
Thus to prove the theorem, it is sufficient to show that the  equation (of $x$)
\begin{equation}\label{}
  f(x)=M(x-\alpha)(x-\beta)\eta^{2}(x)
  +\frac{1-c_{1}}{4} (\frac{1}{2}+l_{\mathfrak{p}})x(x-1)[\eta(x)-
  \frac{1}{2}(\frac{1}{4}+\frac{1}{2}l_{\mathfrak{p}})x]=0,
\end{equation}
 admits a real positive solution $x\neq 1$, $\alpha$ or $\beta$. Now we prove this assertion
 case by case.

\medskip
\noindent \textbf{Case 1} \quad $h_{\mathfrak{n}}<h_{\mathfrak{u}}.$

In this case, $\Delta(1)=M(1-\alpha)(1-\beta)=\frac{1}{2}(h_{\mathfrak{n}}-h_{\mathfrak{u}})<0$,
so we can assume $0<\alpha<1<\beta$ without losing generality. Notice that
\begin{eqnarray*}
  f(0)&=&\frac{1}{4}M\alpha \beta (k_{\mathfrak{p}}-h_{\mathfrak{p}})^{2}>0, \\
  f(1) &=&\frac{1}{4}(\frac{1}{2}+l_{\mathfrak{p}}-h_{\mathfrak{p}}-\frac{c_{1}}{2})^{2}\Delta(1)<0,
  \end{eqnarray*}
  and
   $$\lim_{x\rightarrow +\infty} f(x)=+\infty.$$
 Thus there exist real numbers $z_{1}, z_{2}\in \mathds{R}$ such that $0<z_{1}<1<z_{2}$, and $f(z_{1})=f(z_{2})=0$.
Notice also that the equation
 $\eta(x)-\frac{1}{2}(\frac{1}{4}+\frac{1}{2}l_{\mathfrak{p}})x=0$ admits only one solution,
 and $f(\alpha)$ and $f(\beta)$ can not be equal to zero at the same time. Hence we have either $z_{1}\neq \alpha$ or
 $z_{2}\neq \beta$, which proves the theorem in this case.

\medskip
\noindent \textbf{Case 2} \quad $h_{\mathfrak{n}}=h_{\mathfrak{u}}$, and $G$ is simple.

 In this case, the standard metrics on $G/L$, $G/K$ and $G/H$ are Einstein, and these spaces have been classified in Section 3,  which  are listed in Table A and Table B.

 Clearly, $x=y=1$ is a solution of the system of equations \eqref{equ1}, \eqref{equ2}, and \eqref{equ3}. That is to say,  $x=1$ is a solution of \eqref{Einstein equ}, so one of $\alpha, \beta$ is equal to 1. Without losing generality, we assume
 $\alpha=1$.  Then by \eqref{M},  we have
 $\beta=\frac{1}{M}[\frac{1}{2}(k_{\mathfrak{u}}-h_{\mathfrak{u}})+\frac{1-c_{1}}{4}]
 =\frac{2k_{\mathfrak{p}}-2h_{\mathfrak{p}}+1-c_{2}}{c_{2}+2h_{\mathfrak{p}}}$.
Then can  easily deduce the fact  $\beta>1$ from the proof of Theorem 5.10 of \cite{wz85}.

 Now let
 \begin{eqnarray}\label{app B}
   \bar{f}(x) &=& M(x-\beta)\eta^{2}(x)
   +\frac{1-c_{1}}{4}(\frac{1}{2}+l_{\mathfrak{p}})x[\eta(x)-\frac{1}{2}(\frac{1}{4}+\frac{1}{2}l_{\mathfrak{p}})x] \nonumber\\
    &=&  M(x-\beta)\eta^{2}(x)
    +\frac{1-c_{1}}{8}(\frac{1}{2}+l_{\mathfrak{p}})x(\omega_{1}x+k_{\mathfrak{p}}-h_{\mathfrak{p}}),
 \end{eqnarray}
 where $\omega_{1}=\frac{1}{4}+\frac{1}{2}l_{\mathfrak{p}}-k_{\mathfrak{p}}-\frac{c_{1}}{2}$.

 Clearly, $f(x)=\bar{f}(x)(x-1)$. Thus  the theorem will follow if one can prove that, in this case,  $\bar{f}(x)=0$ admits a real positive solution $x\neq 1, \beta$ except for the three cases (I), (II) and (III).

 First, by the facts that
 $$\bar{f}(0)=-\frac{1}{4}M\beta (k_{\mathfrak{p}}-h_{\mathfrak{p}})^{2}<0,$$
 and
  $$ \lim_{x\rightarrow +\infty}\bar{f}(x)=+\infty,$$
  there exists a unique positive number $z\in \mathds{R}$ such that $\bar{f}(z)=0$.

 Note that if $\omega_{1}\geq 0$, then
 $\bar{f}(\beta)=\frac{1-c_{1}}{8}\beta (\frac{1}{2}+l_{\mathfrak{p}})
 (\omega_{1} \beta+k_{\mathfrak{p}}-h_{\mathfrak{p}})>0$, and so $z<\beta$.

  By Proposition \ref{prop4.4},  $x=z$ is a solution of equation \eqref{equ4.18}.
 Then we have $z>\delta_{0}$, where $\delta_{0}=\frac{2(k_{\mathfrak{u}}-h_{\mathfrak{u}})+c_{1}-c_{2}}{c_{2}+2h_{\mathfrak{n}}}$
 is a solution of the equation $\delta(x)=0$. Now $\delta_{0}\geq 1$ if and only if $\omega_{2}\geq 0$,
 where $\omega_{2}=2k_{\mathfrak{p}}+c_{1}-2c_{2}-4h_{\mathfrak{p}}$. In summarizing, we have
 the following facts:
 \begin{equation}\label{equ4.48}
 \left\{
   \begin{array}{ll}
     z<\beta & \mathrm{if} \quad \omega_{1}\geq 0, \\
     z>1 & \mathrm{if} \quad \omega_{2}\geq 0.
   \end{array}
 \right.
 \end{equation}
Notice that \eqref{equ4.48} is also valid when $G$ is only semisimple.

 Now by Proposition \ref{prop4.5},
 for any standard quadruple
 $(G,L,K,H)$ listed in Table A and Table B, there exists an invariant Einstein metric on $G/H$ of the form
 $g_{(x,y)}$ with $x\neq 1$, $x\neq y$,  except for the cases (a)-(f) therein.
 We will deal with the cases of (a)-(f) listed in Proposition \ref{prop4.5} in Appendix B.

Now the proof of the theorem is completed.
\end{proof}

It is clear that Theorem \ref{main2}  is the second case of this Theorem.

In particular, from Table A, we obtain some new invariant Einstein metrics on some flag manifolds $G/T$, where $G=\mathrm{SU}(n), \mathrm{SO(2n)}$, or $E_{8}$, and $T$ is a maximal compact connected abelian subgroup of $G$.
These Einstein metrics on flag manifolds are clearly neither Kahlerian \cite{besse87book} nor naturally reductive.

We should also mention that, by the above result,  the standard quadruple
$$\bigg(\mathrm{Sp}(8m(9m+1)),(9m+1)\mathrm{Sp}(8m),2(9m+1)\mathrm{Sp}(4m),4(9m+1)\mathrm{Sp}(2m)\bigg), \, m\in\mathds{N}^{+},$$
doesn't  correspond to any  new invariant Einstein metric on the homogeneous space $$\mathrm{Sp}(8m(9m+1))/4(9m+1)\mathrm{Sp}(2m).$$
However, we do find at least two new invariant Einstein metrics on the space associated to the standard quadruples
$$\bigg(\mathrm{Sp}(8m(9m+1)),2\mathrm{Sp}(4m(9m+1)),2(9m+1)\mathrm{Sp}(4m),4(9m+1)\mathrm{Sp}(2m)\bigg)$$
and
$$\bigg(\mathrm{Sp}(8m(9m+1)),2\mathrm{Sp}(4m(9m+1)),4\mathrm{Sp}(2m(9m+1)),4(9m+1)\mathrm{Sp}(2m)\bigg).$$

Finally, we give some new examples of homogeneous Einstein manifolds $G/H$ with $G$ semisimple.

\begin{thm}\label{thm4.7}
Let $G=n_{1}n_{2}n_{3}H$, $L=n_{1}n_{2}H$, $K=n_{1}H$ with $H$ compact simple, where $n_{1}, n_{2}, n_{3}\in \mathds{N}$,
and $n_{i}\geq 2$. Let $H$ be embedded into $G$ by the map
$h\mapsto (h,h,\cdots, h)$. Then
\begin{enumerate}
 \item $(G,L,K,H)$ and $(G,L,K,\{e\})$ are both standard quadruples, and the standard metrics on $G/L$, $G/K$, $G/H$ are Einstein.
\item $G/H$ admits an invariant non-naturally reductive Einstein metric of the form $g_{(x,y)}$ with $x\neq 1, x\neq y$, associated to the quadruple $(G,L,K,H)$.
\end{enumerate}
\end{thm}
\begin{proof}
The first assertion follows from Proposition 5.5 of \cite{wz85}. For the basic quadruple  $(G,L,K,H)$, one has
\begin{gather*}
  c_{1}=\frac{1}{n_{3}}, \quad c_{2}=\frac{1}{n_{2}n_{3}}, \\
  l_{\mathfrak{p}}=\frac{\dim L}{\dim G/L}(1-c_{1})=\frac{1}{n_{3}},\\
  k_{\mathfrak{p}}=\frac{1}{n_{2}n_{3}},\quad h_{\mathfrak{p}}=\frac{1}{n_{1}n_{2}n_{3}}.
\end{gather*}
Then we have
\begin{gather*}
  \omega_{1}=\frac{1}{4}+\frac{1}{2}l_{\mathfrak{p}}-k_{\mathfrak{p}}-\frac{c_{1}}{2}
  =\frac{1}{4}-\frac{1}{n_{2}n_{3}}\geq 0,\\
\omega_{2}=2k_{\mathfrak{p}}+c_{1}-2c_{2}-4h_{\mathfrak{p}}
=\frac{1}{n_{3}}-\frac{4}{n_{1}n_{2}n_{3}}\geq 0.
\end{gather*}
Now the second assertion follows from \eqref{equ4.48} of Theorem \ref{main}.
\end{proof}

Now we can prove
\begin{thm}
Let $H$ be a compact simple Lie group, and $G=H\times H\times \cdots \times H$ ($n$ times, $n\geq 2$), where
 $n=p_{1}^{l_{1}}p_{2}^{l_{2}}\cdots p_{s}^{l_{s}}$, with
$p_{i}$ prime, and $p_{i}\neq p_{j}$, $i\neq j$.
 Then  $G$  admits at least $(l_{1}+1)(l_{2}+1)\cdots (l_{s}+1)-2$ non-equivalent  non-naturally reductive Einstein metrics.
\end{thm}
\begin{proof}
 Given an integer pair $(p,q)$,  denote $n=pq, p,q \geq 2$, and let $L=pH$, $K=H$.  Then $(G,L,K,\{e\})$ is a basic quadruple, and by Theorem \ref{thm4.7},
the standard metrics on $G/L$, $G/K$ are  Einstein.
For the basic quadruple $(G,L,K,\{e\})$, we have
\begin{gather*}
  c_{1}=\frac{1}{q}, \quad c_{2}=\frac{1}{pq}, \\
  l_{\mathfrak{p}}=\frac{\dim L}{\dim G/L}(1-c_{1})=\frac{1}{q},\quad k_{\mathfrak{p}}=\frac{1}{pq}.
\end{gather*}
Then
\begin{gather*}
  \omega_{1}=\frac{1}{4}+\frac{1}{2}l_{\mathfrak{p}}-k_{\mathfrak{p}}-\frac{c_{1}}{2}
  =\frac{1}{4}-\frac{1}{pq}\geq 0,\\
\omega_{2}=2k_{\mathfrak{p}}+c_{1}-2c_{2}-4h_{\mathfrak{p}}
=\frac{1}{q}\geq 0.
\end{gather*}
Thus by \eqref{equ4.48} of Theorem \ref{main},
 $G$ admits a left invariant Einstein metric of the form $g_{(x,y)}$ with $x\neq 1, x\neq y$, associated to $(G,L,K,\{e\})$, which is not naturally reductive.  This completes the proof of the theorem.
\end{proof}

To the best knowledge of the authors, the Einstein metrics on compact semisimple Lie groups described on the above theorem are the first
known examples of  non-naturally reductive Einstein metrics which are not a product of Einstein metrics.

\appendix

\section{The related quantities in the Proof of Proposition 5.1}

In this appendix, we list the calculations of the related quantities in the proof of Proposition 5.1.  This will be described  case by case below.

\medskip
\noindent \textbf{Type A. 1}:
$\mathfrak{su}(n_{1}n_{2}n_{3}k)\supset s(n_{1}\mathfrak{u}(n_{2}n_{3}k))
  \supset  s(n_{1}n_{2}\mathfrak{u}(n_{3}k))\supset  s(n_{1}n_{2}n_{3}\mathfrak{u}(k))$,\quad $ k\geq1, n_{i}\geq 2$.

In this case, we have
\begin{gather*}
   c_{1}=\frac{1}{n_{1}},\quad c_{2}=\frac{1}{n_{1}n_{2}}.
\end{gather*}
Then by Lemma \ref{lem3.2}, we have
\begin{equation*}
  l_{\mathfrak{p}}=\frac{1}{\dim G/L}[n_{1}\dim \mathfrak{su}(n_{2}n_{3}k)(1-c_{1})+(n_{1}-1)]=\frac{1}{n_{1}},
\end{equation*}
and similarly
\begin{equation*}
  k_{\mathfrak{p}}=\frac{1}{n_{1}n_{2}}, \quad h_{\mathfrak{p}}=\frac{1}{n_{1}n_{2}n_{3}}.
\end{equation*}
Therefore we have
\begin{gather*}
  \omega_{1}=\frac{1}{4}+\frac{1}{2}l_{\mathfrak{p}}-k_{\mathfrak{p}}-\frac{c_{1}}{2}
  =\frac{1}{4}-\frac{1}{n_{1}n_{2}}\geq0,\\
   \omega_{2}=2k_{\mathfrak{p}}+c_{1}-2c_{2}-4h_{\mathfrak{p}}
   =\frac{1}{n_{1}}-\frac{4}{n_{1}n_{2}n_{3}}\geq0.
\end{gather*}

\medskip
\noindent\textbf{Type A. 2}:
$\mathfrak{so}(n_{1}n_{2}n_{3}k)\supset n_{1}\mathfrak{so}(n_{2}n_{3}k)
  \supset  n_{1}n_{2}\mathfrak{so}(n_{3}k)\supset  n_{1}n_{2}n_{3}\mathfrak{so}(k)$, \quad $k\geq2, n_{i}\geq2$.

In this case,
we have
\begin{gather*}
  c_{1}=\frac{n_{2}n_{3}k-2}{n_{1}n_{2}n_{3}k-2}, \quad c_{2}=\frac{n_{3}k-2}{n_{1}n_{2}n_{3}k-2}.
\end{gather*}
Then
\begin{eqnarray*}
  l_{\mathfrak{p}} &=& \frac{\dim L}{\dim G/L}(1-c_{1}) \\
   &=& \frac{n_{1}n_{2}n_{3}k(n_{2}n_{3}k-1)}{n_{1}n_{2}n_{3}k(n_{1}n_{2}n_{3}k-1)-n_{1}n_{2}n_{3}k(n_{2}n_{3}k-1)}
   (1-\frac{n_{2}n_{3}k-2}{n_{1}n_{2}n_{3}k-2}) \\
   &=&  \frac{n_{2}n_{3}k-1}{n_{1}n_{2}n_{3}k-2},
\end{eqnarray*}
and similarly
\begin{gather*}
   k_{\mathfrak{p}}=\frac{n_{3}k-1}{n_{1}n_{2}n_{3}k-2},\quad
   h_{\mathfrak{p}}=\frac{k-1}{n_{1}n_{2}n_{3}k-2}.
\end{gather*}
Therefore we have
\begin{gather*}
 \omega_{1}=\frac{1}{4}+\frac{1}{2}l_{\mathfrak{p}}-k_{\mathfrak{p}}-\frac{c_{1}}{2}
 =\frac{n_{1}n_{2}n_{3}k-4n_{3}k+4}{4(n_{1}n_{2}n_{3}k-2)}>0, \\ \omega_{2}=2k_{\mathfrak{p}}+c_{1}-2c_{2}-4h_{\mathfrak{p}}=\frac{n_{2}n_{3}k-4k+4}{n_{1}n_{2}n_{3}k-2}>0.
\end{gather*}

\medskip
\noindent\textbf{Type A. 3}:
$\mathfrak{so}(n_{1}n_{2}k)\supset n_{1}\mathfrak{so}(n_{2}k)
  \supset  n_{1}n_{2}\mathfrak{so}(k)\supset \oplus_{i=1}^{l}\mathfrak{h}_{i}$,  \quad $k\geq3, n_{i}\geq2$.

In this case, we have
\begin{equation*}
  c_{1}=\frac{n_{2}k-2}{n_{1}n_{2}k-2}, \quad c_{2}=\frac{k-2}{n_{1}n_{2}k-2},
  \quad l_{\mathfrak{p}}=\frac{n_{2}k-1}{n_{1}n_{2}k-2}.
\end{equation*}
Therefore
\begin{equation*}
 \omega_{1}=\frac{n_{1}n_{2}k-4k+4}{4(n_{1}n_{2}k-2)}>0.
\end{equation*}
Moreover,  by \eqref{4.41}, we have
\begin{eqnarray*}
  \omega_{2} &=& \frac{\dim G+\dim L-4\dim H}{\dim L}l_{\mathfrak{p}}-1 \\
   &> & \frac{\dim G+\dim L-2\dim K}{\dim L}l_{\mathfrak{p}}-1 \\
   &=&  \frac{n_{1}n_{2}k(n_{1}n_{2}k-1)+n_{1}n_{2}k(n_{2}k-1)-2n_{1}n_{2}k(k-1)}{n_{1}n_{2}k(n_{2}k-1)}\times
       \frac{n_{2}k-1}{n_{1}n_{2}k-2} -1\\
   &=&  \frac{n_{2}k-2k+2}{n_{1}n_{2}k-2} \\
   &>&0,
\end{eqnarray*}
since $\dim H<\frac{1}{2}\dim K$.

\medskip
\noindent \textbf{Type A. 4}:
$\mathfrak{sp}(n_{1}n_{2}n_{3}k)\supset n_{1}\mathfrak{sp}(n_{2}n_{3}k)
  \supset  n_{1}n_{2}\mathfrak{sp}(n_{3}k)\supset  n_{1}n_{2}n_{3}\mathfrak{sp}(k)$, \quad $k\geq1, n_{i}\geq 2$.

In this case, we have
\begin{equation*}
   c_{1}=\frac{n_{2}n_{3}k+1}{n_{1}n_{2}n_{3}k+1},\quad c_{2}=\frac{n_{3}k+1}{n_{1}n_{2}n_{3}k+1}.
\end{equation*}
Thus
\begin{eqnarray*}
  l_{\mathfrak{p}} &=& \frac{\dim L}{\dim G/L}(1-c_{1}) \\
   &=& \frac{n_{1}n_{2}n_{3}k(2n_{2}n_{3}k+1)}{n_{1}n_{2}n_{3}k(2n_{1}n_{2}n_{3}k+1)-n_{1}n_{2}n_{3}k(2n_{2}n_{3}k+1)}
   (1-\frac{n_{2}n_{3}k+1}{n_{1}n_{2}n_{3}k+1}) \\
   &=&  \frac{2n_{2}n_{3}k+1}{2(n_{1}n_{2}n_{3}k+1)},
\end{eqnarray*}
and similarly
\begin{gather*}
  k_{\mathfrak{p}}=\frac{2n_{3}k+1}{2(n_{1}n_{2}n_{3}k+1)},\quad
  h_{\mathfrak{p}}=\frac{2k+1}{2(n_{1}n_{2}n_{3}k+1)}.
\end{gather*}
Therefore
\begin{gather*}
  \omega_{1}=\frac{1}{4}+\frac{1}{2}l_{\mathfrak{p}}-k_{\mathfrak{p}}-\frac{c_{1}}{2}
  =\frac{n_{1}n_{2}n_{3}k-4n_{3}k-2}{4(n_{1}n_{2}n_{3}k+1)},\\
  \omega_{2}=2k_{\mathfrak{p}}+c_{1}-2c_{2}-4h_{\mathfrak{p}}=\frac{n_{2}n_{3}k-4k-2}{n_{1}n_{2}n_{3}k+1}.
\end{gather*}
It follows that $\omega_{1}<0$ if and only if $n_{1}=n_{2}=2$, $\omega_{2}<0$, if and only if $n_{2}=n_{3}=2$.

\medskip
\noindent\textbf{Type A. 5}: $
  \mathfrak{e}_{6}\supset \mathfrak{so}(10)\oplus \mathds{R}
  \supset \mathfrak{so}(8)\oplus \mathds{R}^{2}\supset \mathds{R}^{6}$.

Note that $\mathfrak{so}(10)$ and $\mathfrak{so}(8)$ are regular subalgebras of $\mathfrak{e}_{6}$, hence we have $c_{1}=\frac{2}{3}$, $c_{2}=\frac{1}{2}$.
Since $G/L$ and $\bar{L}/\bar{K}$ are both symmetric, we have $l_{\mathfrak{p}}=\frac{1}{2}$,
$k_{\mathfrak{p}}=\frac{1}{3}$, and $h_{\mathfrak{p}}=\frac{1}{12}$.
Thus
\begin{gather*}
  \omega_{1}=\frac{1}{4}+\frac{1}{2}l_{\mathfrak{p}}-k_{\mathfrak{p}}-\frac{c_{1}}{2}
  =\frac{1}{4}+\frac{1}{2}\times \frac{1}{2}-\frac{1}{3}-\frac{1}{2}\times \frac{2}{3}=-\frac{1}{6}, \\
  \omega_{2}=2k_{\mathfrak{p}}+c_{1}-2c_{2}-4h_{\mathfrak{p}}
  =2\times \frac{1}{3}+\frac{2}{3}-2\times \frac{1}{2}-4\times \frac{1}{12}=0.
\end{gather*}

\medskip
\noindent \textbf{Type A. 6}:
$ \mathfrak{e}_{7}\supset \mathfrak{so}(12)\oplus \mathfrak{su}(2)
  \supset \mathfrak{so}(8)\oplus 3\mathfrak{su}(2)\supset 7\mathfrak{su}(2)$.

Note that $\mathfrak{so}(12)$, $\mathfrak{so}(8)$ and $7\mathfrak{su}(2)$ are regular subalgebras of $\mathfrak{e}_{7}$,
hence we have $c_{1}=\frac{5}{9}$, $c_{2}=\frac{1}{3}$, and $B_{\mathfrak{su}(2)}=\frac{1}{9}B|_{\mathfrak{su}(2)}$.

Since $G/L$ and $\bar{L}/\bar{K}$ are both symmetric, we have $l_{\mathfrak{p}}=\frac{1}{2}$,
$k_{\mathfrak{p}}=\frac{5}{18}$, and $h_{\mathfrak{p}}=\frac{1}{6}$. Therefore
\begin{gather*}
  \omega_{1}=\frac{1}{4}+\frac{1}{2}l_{\mathfrak{p}}-k_{\mathfrak{p}}-\frac{c_{1}}{2}
  =\frac{1}{4}+\frac{1}{2}\times \frac{1}{2}-\frac{5}{18}-\frac{1}{2}\times \frac{5}{9}=-\frac{1}{18}, \\
  \omega_{2}=2k_{\mathfrak{p}}+c_{1}-2c_{2}-4h_{\mathfrak{p}}
  =2\times \frac{5}{18}+\frac{5}{9}-2\times \frac{1}{3}-4\times \frac{1}{6}=-\frac{2}{9}.
\end{gather*}

\medskip
\noindent\textbf{Type A. 9}:
$\mathfrak{e}_{8}\supset 2\mathfrak{so}(8)
  \supset 8\mathfrak{su}(2)\supset \mathds{R}^{8}$.

Note that $2\mathfrak{so}(8)$ and $8\mathfrak{su}(8)$ are regular subalgebras of $\mathfrak{e}_{8}$, hence we have
\begin{equation*}
  c_{1}=\frac{1}{5},\quad l_{\mathfrak{p}}=\frac{7}{30}.
\end{equation*}
It follows that
\begin{eqnarray*}
  \omega_{1} &=& -\frac{1}{4}+  \frac{\dim G-2\dim K}{2\dim L}l_{\mathfrak{p}} \\
   &=& -\frac{1}{4}+ \frac{248-2\times 24}{2\times 56}\times \frac{7}{30} \\
   &=&  \frac{1}{6},
\end{eqnarray*}
and
\begin{eqnarray*}
  \omega_{2} &=& \frac{\dim G+\dim L-4\dim H}{\dim L}l_{\mathfrak{p}}-1 \\
   &=&  \frac{248+56-4\times 8}{ 56}\times \frac{7}{30}-1 \\
   &=&  \frac{2}{15}.
\end{eqnarray*}

\medskip
Now we deal with the cases of Type B.
Notice that for any standard quadruple $(G,L,K,H)$ listed in Table B with $H=\{e\}$, one has
$h_{\mathfrak{n}}=h_{\mathfrak{u}}=h_{\mathfrak{p}}=0$.

\medskip
\noindent\textbf{Type B. 1}:
$\mathfrak{so}(n_{1}n_{2}k)\supset n_{1}\mathfrak{so}(n_{2}k)
  \supset  n_{1}n_{2}\mathfrak{so}(k), \quad k\geq3, n_{i}\geq 2$.

In this case, we have
\begin{gather*}
  c_{1}=\frac{n_{2}k-2}{n_{1}n_{2}k-2},\quad c_{2}=\frac{k-2}{n_{1}n_{2}k-2}, \\
  l_{\mathfrak{p}}=\frac{n_{2}k-1}{n_{1}n_{2}k-2},\quad k_{\mathfrak{p}}=\frac{k-1}{n_{1}n_{2}k-2}.
\end{gather*}
It follows that
$$
  \omega_{1}=\frac{1}{4}+\frac{1}{2}l_{\mathfrak{p}}-k_{\mathfrak{p}}-\frac{c_{1}}{2}
  =\frac{n_{1}n_{2}k-4k+4}{4(n_{1}n_{2}k-2)}>0, $$
  and
  $$\omega_{2}=2k_{\mathfrak{p}}+c_{1}-2c_{2}-4h_{\mathfrak{p}}=\frac{n_{2}k}{n_{1}n_{2}k-2}.$$

\medskip
\noindent \textbf{Type B. 2}: $
  \mathfrak{so}(nk)\supset n\mathfrak{so}(k)
  \supset \oplus_{i=1}^{l}\mathfrak{h}_{i}$, \quad $k\geq3, n\geq 2$.

It is easily seen that
\begin{equation*}
  c_{1}=\frac{k-2}{nk-2},\quad l_{\mathfrak{p}}=\frac{k-1}{nk-2}.
\end{equation*}
since $\dim K<\frac{1}{2}\dim L$, we have
\begin{eqnarray*}
  \omega_{1} &=&  -\frac{1}{4}+  \frac{\dim G-2\dim K}{2\dim L}l_{\mathfrak{p}} \\
   &>& -\frac{1}{4}+  \frac{\dim G-\dim L}{2\dim L}l_{\mathfrak{p}} \\
   &=& -\frac{1}{4}+\frac{nk(nk-1)-nk(k-1)}{2nk(k-1)}\times \frac{k-1}{nk-2}\\
   &=&\frac{nk-2k+2}{4(nk-2)}\\
   &>&0.
\end{eqnarray*}
On the other hand, we have
\begin{eqnarray*}
  \omega_{2} &=& \frac{\dim G+\dim L-4\dim H}{\dim L}l_{\mathfrak{p}}-1 \\
   &=&  \frac{nk(nk-1)+nk(k-1)}{nk(k-1)}\times \frac{k-1}{nk-2}-1 \\
   &=&  \frac{k}{nk-2}.
\end{eqnarray*}

\medskip
\noindent\textbf{Type B. 3}:
$
  \mathfrak{sp}(n_{1}n_{2}k)\supset n_{1}\mathfrak{sp}(n_{2}k)
  \supset  n_{1}n_{2}\mathfrak{sp}(k)$, \quad $k\geq1, n_{i}\geq2$.

In this case, we have
\begin{gather*}
  c_{1}=\frac{n_{2}k+1}{n_{1}n_{2}k+1},\quad c_{2}=\frac{k+1}{n_{1}n_{2}k+1}, \\
  l_{\mathfrak{p}}=\frac{2n_{2}k+1}{2(n_{1}n_{2}k+1)}, \quad
  k_{\mathfrak{p}}=\frac{2k+1}{2(n_{1}n_{2}k+1)}.
\end{gather*}
It follows that
\begin{eqnarray*}
  \omega_{1} &=& \frac{1}{4}+\frac{1}{2}l_{\mathfrak{p}}-k_{\mathfrak{p}}-\frac{c_{1}}{2} \\
   &=& \frac{1}{4}+\frac{1}{2}\times \frac{2n_{2}k+1}{2(n_{1}n_{2}k+1)}
   -\frac{2k+1}{2(n_{1}n_{2}k+1)}-\frac{1}{2}\times\frac{n_{2}k+1}{n_{1}n_{2}k+1} \\
   &=&  \frac{n_{1}n_{2}k-4k-2}{4(n_{1}n_{2}k+1)},
\end{eqnarray*}
and
\begin{eqnarray*}
  \omega_{2} &=& 2k_{\mathfrak{p}}+c_{1}-2c_{2}-4h_{\mathfrak{p}} \\
   &=&  2\times\frac{2k+1}{2(n_{1}n_{2}k+1)}+\frac{n_{2}k+1}{n_{1}n_{2}k+1}
   -2\times\frac{k+1}{n_{1}n_{2}k+1}\\
   &=&  \frac{n_{2}k}{n_{1}n_{2}k+1}\\
   &>&0.
  \end{eqnarray*}
It is easily seen that $\omega_{1}<0$ if and only if $n_{1}=n_{2}=2$.

\medskip
\noindent \textbf{Type B. 4}: $  \mathfrak{so}(8)\supset \mathfrak{so}(7)\supset \mathfrak{g}_{2}$.

Since $\mathrm{SO}(8)/\mathrm{SO}(7)$ is symmetric, we have $l_{\mathfrak{p}}=\frac{1}{2}$.
By \eqref{omega1} and \eqref{omega2}, we have
\begin{eqnarray*}
  \omega_{1} &=& \frac{1}{4\dim L}(\dim G-\dim L-2\dim K) \\
   &=& \frac{1}{4\times 21}(28-21-2\times 14) \\
   &=&  -\frac{1}{4},
\end{eqnarray*}
and
\begin{eqnarray*}
  \omega_{2} &=& \frac{1}{2\dim L}(\dim G-\dim L) \\
   &=& \frac{1}{2\times 21}(28-21) \\
   &=&  \frac{1}{6}.
\end{eqnarray*}

\medskip
\noindent \textbf{Type B. 5}:
  $\mathfrak{f}_{4}\supset \mathfrak{so}(9)\supset \mathfrak{so}(8)$.

Since $F_{4}/\mathrm{SO}(9)$ is symmetric, we have $l_{\mathfrak{p}}=\frac{1}{2}$.
By \eqref{omega1} and \eqref{omega2}, we have
\begin{eqnarray*}
  \omega_{1} &=& \frac{1}{4\dim L}(\dim G-\dim L-2\dim K) \\
   &=& \frac{1}{4\times 36}(52-36-2\times 28) \\
   &=&  -\frac{5}{18},
\end{eqnarray*}
and
\begin{eqnarray*}
  \omega_{2} &=& \frac{1}{2\dim L}(\dim G-\dim L) \\
   &=& \frac{1}{2\times 36}(52-36) \\
   &=&  \frac{2}{9}.
\end{eqnarray*}

\medskip
\noindent\textbf{Type B. 6}:
$
  \mathfrak{e}_{6}\supset 3\mathfrak{su}(3)\supset 3\mathfrak{so}(3)$.

Note that $3\mathfrak{su}(3)$ is a regular subalgebra of $\mathfrak{e}_{6}$, hence we have
\begin{gather*}
  c_{1}=\frac{1}{4},\quad c_{2}=\frac{1}{24}, \\
  l_{\mathfrak{p}}=\frac{3\times 8}{78-3\times 8}(1-\frac{1}{4})=\frac{1}{3},\\
   k_{\mathfrak{p}}=\frac{3\times 3}{3\times 8}\times \frac{1}{3}=\frac{1}{8}.
\end{gather*}
Therefore
\begin{gather*}
  \omega_{1}=\frac{1}{4}+\frac{1}{2}l_{\mathfrak{p}}-k_{\mathfrak{p}}-\frac{c_{1}}{2}
  =\frac{1}{4}+\frac{1}{2}\times \frac{1}{3}-\frac{1}{8}-\frac{1}{2}\times \frac{1}{4}=\frac{1}{6}, \\
   \omega_{2}=2k_{\mathfrak{p}}+c_{1}-2c_{2}-4h_{\mathfrak{p}}
   =2\times \frac{1}{8}+\frac{1}{4}-2\times \frac{1}{24}=\frac{5}{12}.
\end{gather*}

\medskip
\noindent\textbf{Type B. 7}:
$
  \mathfrak{e}_{7}\supset \mathfrak{su}(8)\supset \mathfrak{so}(8)$.

Since $E_{7}/\mathrm{SU}(8)$ is symmetric, we have $l_{\mathfrak{p}}=\frac{1}{2}$.
By \eqref{omega1} and \eqref{omega2}, we get
\begin{eqnarray*}
  \omega_{1} &=& \frac{1}{4\dim L}(\dim G-\dim L-2\dim K) \\
   &=& \frac{1}{4\times 63}(133-63-2\times 28) \\
   &=&  \frac{1}{18},
\end{eqnarray*}
and
\begin{eqnarray*}
  \omega_{2} &=& \frac{1}{2\dim L}(\dim G-\dim L) \\
   &=& \frac{1}{2\times 63}(133-63) \\
   &=&  \frac{5}{9}.
\end{eqnarray*}

\medskip
\noindent\textbf{Type B. 13}:
$\mathfrak{e}_{8}\supset \mathfrak{su}(9)\supset \mathfrak{so}(9)$, and \textbf{Type B. 14}:
   $\mathfrak{e}_{8}\supset \mathfrak{su}(9)\supset 2\mathfrak{su}(3)$.

Clearly, $G/L=E_{8}/\mathrm{SU}(9)$, $\mathfrak{su}(9)$ is a regular subalgebra of $\mathfrak{e}_{8}$, hence we have $c_{1}=\frac{3}{10}$,
and $l_{\mathfrak{p}}=\frac{80}{248-80}\times (1-\frac{3}{10})=\frac{1}{3}$.
Since $\dim 2\mathfrak{su}(3)<\dim \mathfrak{so}(9)=36$, we have
\begin{eqnarray*}
  \omega_{1} &=& -\frac{1}{4}+  \frac{\dim G-2\dim K}{2\dim L}l_{\mathfrak{p}} \\
   &\geq&  -\frac{1}{4}+ \frac{248-2\times 36}{2\times 80}\times \frac{1}{3}\\
   &=& \frac{7}{60},
\end{eqnarray*}
and
\begin{eqnarray*}
  \omega_{2} &=& \frac{\dim G+\dim L-4\dim H}{\dim L}l_{\mathfrak{p}}-1 \\
   &=& \frac{248+80}{80}\times \frac{1}{3}-1 \\
   &=&  \frac{11}{30}.
\end{eqnarray*}

\medskip
\noindent \textbf{Type B. 15}:
  $\mathfrak{e}_{8}\supset 2\mathfrak{so}(8)\supset 8\mathfrak{su}(2)$, and \textbf{Type B. 16}:
  $\mathfrak{e}_{8}\supset 2\mathfrak{so}(8)\supset 2\mathfrak{su}(3)$.

Clearly, $G/L=E_{8}/\mathrm{SO}(8)\times \mathrm{SO}(8)$,
$2\mathfrak{so}(8)$ is a regular subalgebra of $\mathfrak{e}_{8}$,
hence we have  $c_{1}=\frac{1}{5}$,
and $l_{\mathfrak{p}}=\frac{56}{248-56}(1-\frac{1}{5})=\frac{7}{30}$.
Since $\dim 2\mathfrak{su}(3)<\dim 8\mathfrak{su}(2)=24$, we have
\begin{eqnarray*}
  \omega_{1} &=& -\frac{1}{4}+  \frac{\dim G-2\dim K}{2\dim L}l_{\mathfrak{p}} \\
   &\geq&  -\frac{1}{4}+ \frac{248-2\times 24}{2\times 56}\times \frac{7}{30}\\
   &=& \frac{1}{6},
\end{eqnarray*}
and
\begin{eqnarray*}
  \omega_{2} &=& \frac{\dim G+\dim L-4\dim H}{\dim L}l_{\mathfrak{p}}-1 \\
   &=& \frac{248+56}{56}\times \frac{7}{30}-1 \\
   &=&  \frac{4}{15}.
\end{eqnarray*}

\medskip
\textbf{Type B. 17}: $
  \mathfrak{e}_{8}\supset 2\mathfrak{su}(5)\supset 2\mathfrak{so}(5)$.

Note that $2\mathfrak{su}(5)$ is a regular subalgebra of $\mathfrak{e}_{8}$,
$\mathrm{SU}(5)/\mathrm{SO}(5)$ is symmetric, hence we have
\begin{gather*}
  c_{1}=\frac{1}{6}, \quad c_{2}=\frac{1}{20}, \\
  l_{\mathfrak{p}}=\frac{2\times 24}{248-2\times 24}=\frac{1}{5},\\
   k_{\mathfrak{p}}=\frac{20}{2\times 24}\times \frac{1}{5}=\frac{1}{12}.
\end{gather*}
Therefore
\begin{gather*}
  \omega_{1}=\frac{1}{4}+\frac{1}{2}l_{\mathfrak{p}}-k_{\mathfrak{p}}-\frac{c_{1}}{2}
  =\frac{1}{4}+\frac{1}{2}\times \frac{1}{5}-\frac{1}{12}-\frac{1}{2}\times \frac{1}{6}=\frac{11}{60}, \\
   \omega_{2}=2k_{\mathfrak{p}}+c_{1}-2c_{2}-4h_{\mathfrak{p}}
   =2\times \frac{1}{12}+\frac{1}{6}-2\times \frac{1}{20}=\frac{7}{30}.
\end{gather*}

\medskip
\noindent\textbf{Type B. 18}:
$  \mathfrak{e}_{8}\supset 4\mathfrak{su}(3)\supset 4\mathfrak{so}(3)$.

Note that $4\mathfrak{su}(3)$ is a regular subalgebra of $\mathfrak{e}_{8}$,
$\mathrm{SU}(3)/\mathrm{SO}(3)$ is symmetric, hence we have
\begin{gather*}
  c_{1}=\frac{1}{10}, \quad c_{2}=\frac{1}{60}, \\
  l_{\mathfrak{p}}=\frac{4\time 8}{248-4\times 8}(1-\frac{1}{10})=\frac{2}{15},\\
   k_{\mathfrak{p}}=\frac{4\times 3}{4\times 8}\times \frac{2}{15}=\frac{1}{20}.
\end{gather*}
Therefore
\begin{gather*}
  \omega_{1}=\frac{1}{4}+\frac{1}{2}l_{\mathfrak{p}}-k_{\mathfrak{p}}-\frac{c_{1}}{2}
  =\frac{1}{4}+\frac{1}{2}\times \frac{2}{15}-\frac{1}{20}-\frac{1}{2}\times \frac{1}{10}=\frac{13}{60}, \\
   \omega_{2}=2k_{\mathfrak{p}}+c_{1}-2c_{2}-4h_{\mathfrak{p}}
   =2\times \frac{1}{20}+\frac{1}{10}-2\times \frac{1}{60}=\frac{1}{6}.
\end{gather*}

\section{The values $\bar{f}(1)$ and $\bar{f}(\beta)$ in the proof of Theorem 5.2}
In this appendix, we list the values $\bar{f}(1)$ and $\bar{f}(\beta)$ in the proof of Theorem 5.2. First recall the formula  \eqref{app B}
\begin{eqnarray*}
   \bar{f}(x) &=&  M(x-\beta)\eta^{2}(x)
    +\frac{1-c_{1}}{8}(\frac{1}{2}+l_{\mathfrak{p}})x(\omega_{1}x+k_{\mathfrak{p}}-h_{\mathfrak{p}}),
 \end{eqnarray*}
 where $M=\frac{1}{4}c_{2}+\frac{1}{2}h_{\mathfrak{n}}>0$,
 $\omega_{1}=\frac{1}{4}+\frac{1}{2}l_{\mathfrak{p}}-k_{\mathfrak{p}}-\frac{c_{1}}{2}$.
$z$ is the  unique positive number  such that $\bar{f}(z)=0$.

 Now we  compute the values $\bar{f}(1)$ and $\bar{f}(\beta)$
 of the cases  (a)-(f) listed in Proposition \ref{prop4.5}. This will be completed case by case below.

\medskip
\noindent  \textbf{Case (a)}\quad Type A. 4 with $n_{1}=n_{2}=2$, namely,
 $
  \mathfrak{sp}(4n_{3}k)\supset 2\mathfrak{sp}(2n_{3}k) \supset 4\mathfrak{sp}(n_{3}k)\supset 4n_{3}\mathfrak{sp}(k), \quad k\geq1, n_{3}\geq 2.$

In this case, it is easily seen that
\begin{eqnarray*}
  c_{1}&=&\frac{2n_{3}k+1}{4n_{3}k+1},\quad c_{2}=\frac{n_{3}k+1}{4n_{3}k+1}, \\
  l_{\mathfrak{p}}&=&\frac{1}{2},\quad k_{\mathfrak{p}}=\frac{2n_{3}k+1}{2(4n_{3}k+1)},
  \quad h_{\mathfrak{p}}=\frac{2k+1}{2(4n_{3}k+1)},
  \end{eqnarray*}
  and
  $$\omega_{1}=-\frac{1}{2(4n_{3}k+1)},\quad \omega_{2}=\frac{2n_{3}k-4k-2}{4n_{3}k+1}.$$
Therefore we have
 \begin{equation*}
   \beta=\frac{2k_{\mathfrak{p}}-2h_{\mathfrak{p}}+1-c_{2}}{c_{2}+
   2h_{\mathfrak{p}}}=\frac{5n_{3}k-2k}{n_{3}k+2k+2},
 \end{equation*}
 and
 \begin{eqnarray*}
   \bar{f}(\beta) &=& \frac{1-c_{1}}{8}\beta (\omega_{1}\beta+k_{\mathfrak{p}}-h_{\mathfrak{p}}) \\
    &=& \frac{1-c_{1}}{8}\beta [-\frac{1}{2(4n_{3}k+1)}\times \frac{5n_{3}k-2k}{n_{3}k+2k+2}+\frac{2n_{3}k-2k}{2(4n_{3}k+1)}] \\
    &=& \frac{1-c_{1}}{8}\beta \times \frac{2k(n_{3}-1)(n_{3}k+2k+2)-k(5n_{3}-2)}{2(4n_{3}k+1)(n_{3}k+2k+2)}
    > 0.
 \end{eqnarray*}
It is clear that $\omega_{2}<0$ if and only if $n_{3}=2$. On the other hand, if $n_{3}=2$, then we have
 \begin{eqnarray*}
   \bar{f}(1) &=& M(1-\beta)\eta^{2}(1)+\frac{1-c_{1}}{8}(\omega_{1}+k_{\mathfrak{p}}-h_{\mathfrak{p}})\\
    &=& (\frac{1}{4}\times \frac{2k+1}{8k+1}+\frac{1}{2}\times \frac{2k+1}{2(8k+1)})(1-\frac{4k}{2k+1})
    \frac{1}{4}(1-\frac{2k+1}{2(8k+1)}-\frac{1}{2}\times\frac{4k+1}{8k+1})^{2}\\
    &&+\frac{1}{8}(1-\frac{4k+1}{8k+1})\big[-\frac{1}{2(8k+1)}+\frac{4k+1}{2(8k+1)}-\frac{2k+1}{2(8k+1)}\big] \\
    &=&  \frac{1}{8}\times \frac{2k+1}{8k+1}\times  \frac{1-2k}{2k+1} \times (\frac{5k}{8k+1})^{2}
    +\frac{1}{8}\times \frac{4k}{8k+1}\times \frac{2k-1}{2(8k+1)}\\
    &=& \frac{k(2k-1)(2-9k)}{8(8k+1)^{3}}<0.
 \end{eqnarray*}
 Thus $1<z<\beta$.

In the  case $n_{2}=n_{3}=2$, we have
\begin{equation*}
  \mathfrak{sp}(4n_{1}k)\supset n_{1}\mathfrak{sp}(4k) \supset 2n_{1}\mathfrak{sp}(2k)\supset 4n_{1}\mathfrak{sp}(k),
  \quad k\geq 1, n_{1}\geq 2.
\end{equation*}
It follows that
\begin{eqnarray*}
  c_{1}&=&\frac{4k+1}{4n_{1}k+1},\quad\qquad c_{2}=\frac{2k+1}{4n_{1}k+1}, \\
  l_{\mathfrak{p}}&=&\frac{8k+1}{2(4n_{1}k+1)},\quad k_{\mathfrak{p}}=\frac{4k+1}{2(4n_{1}k+1)},
  \quad h_{\mathfrak{p}}=\frac{2k+1}{2(4n_{1}k+1)},
  \end{eqnarray*}
  and
  $$\omega_{1}=\frac{2n_{1}k-4k-1}{2(4n_{1}k+1)},\quad \omega_{2}=-\frac{2}{4n_{1}k+1}.$$
Then we have
\begin{equation*}
   \beta=\frac{2k_{\mathfrak{p}}-2h_{\mathfrak{p}}+1-c_{2}}{c_{2}+2h_{\mathfrak{p}}}=\frac{2n_{1}k}{2k+1}.
 \end{equation*}
  Notice that the inequality $\omega_{1}<0$ holds only when $n_{1}=2$, and we have studied this case  in the above. Therefore in the following we assume that   $z<\beta$.
 Now
  \begin{eqnarray*}
   \bar{f}(1) &=& M(1-\beta)\eta^{2}(1)
   +\frac{1-c_{1}}{8}(\omega_{1}+k_{\mathfrak{p}}-h_{\mathfrak{p}})(\frac{1}{2}+l_{\mathfrak{p}})\\
    &=& (\frac{1}{4}\times \frac{2k+1}{4n_{1}k+1}+\frac{1}{2}\times \frac{2k+1}{2(4n_{1}k+1)})(1-\frac{2n_{1}k}{2k+1})\times\\
    &&\frac{1}{4}\times
    (\frac{1}{2}+\frac{8k+1}{2(4n_{1}k+1)}-\frac{2k+1}{2(4n_{1}k+1)}-\frac{1}{2}\times\frac{4k+1}{4n_{1}k+1})^{2}\\
    &&+\frac{1}{8}(1-\frac{4k+1}{4n_{1}k+1})
    (\frac{2n_{1}k-4k-1}{2(4n_{1}k+1)}+\frac{4k+1}{2(4n_{1}k+1)}-\frac{2k+1}{2(4n_{1}k+1)})
    (\frac{1}{2}+\frac{8k+1}{2(4n_{1}k+1)}) \\
    &=&\frac{1}{8}\times \frac{2k+1}{4n_{1}k+1}\times \frac{2k+1-2n_{1}k}{2k+1}\times
    (\frac{2n_{1}k+k}{4n_{1}k+1})^{2}\\
    &&+\frac{1}{8}\times \frac{4n_{1}k-4k}{4n_{1}k+1}\times \frac{2n_{1}k-2k-1}{2(4n_{1}k+1)}
    \times \frac{2n_{1}k+4k+1}{4n_{1}k+1}\\
    &=& \frac{k(2n_{1}k-2k-1)}{8(4n_{1}k+1)^{3}}[2(n_{1}-1)(2n_{1}k+4k+1)-k(2n_{1}+1)^{2}]\\
    &=& \frac{k(2n_{1}k-2k-1)}{8(4n_{1}k+1)^{3}}[2(n_{1}-1)-9k].
 \end{eqnarray*}
Thus $\bar{f}(1)=0$ if and only if
\begin{gather*}\label{}
 2(n_{1}-1)-9k=0.
\end{gather*}
Since $k, n_{1}\in \mathds{N}^{+}, n_{1}\geq 2$,  it is clear that $\bar{f}(1)=0$ if and only if
$n_{1}=9m+1$, $k=2m$, where $m\in \mathds{N}^{+}$.

\medskip
\noindent \textbf{Case (b)}\quad Type A. 5:\quad
$\mathfrak{e}_{6}\supset \mathfrak{so}(10)\oplus \mathds{R}
  \supset \mathfrak{so}(8)\oplus \mathds{R}^{2}\supset \mathds{R}^{6}.$

In this case, we have
\begin{eqnarray*}
  c_{1}&=&\frac{2}{3},\quad c_{2}=\frac{1}{2}, \\
  l_{\mathfrak{p}}&=&\frac{1}{2}, \quad k_{\mathfrak{p}}=\frac{1}{3}, \quad h_{\mathfrak{p}}=\frac{1}{12},\\
  \end{eqnarray*}
  and
  $$\omega_{1}=-\frac{1}{6}, \quad \omega_{2}=0.$$
 Then we have
 \begin{equation*}
   \beta=\frac{2k_{\mathfrak{p}}-2h_{\mathfrak{p}}+1-c_{2}}{c_{2}+2h_{\mathfrak{p}}}=\frac{3}{2},
 \end{equation*}
 and
 \begin{eqnarray*}
   \bar{f}(\beta) &=& \frac{1-c_{1}}{8}\beta (\omega_{1}\beta+k_{\mathfrak{p}}-h_{\mathfrak{p}}) \\
    &=& \frac{1}{3}\times \frac{1}{8}\times \frac{3}{2}(-\frac{1}{6}\times \frac{3}{2}+\frac{1}{3}-\frac{1}{12}) \\
    &=& 0.
 \end{eqnarray*}
 So $x=\beta=\frac{3}{2}$ is the only real solution of $\bar{f}(x)=0$.

\medskip
\noindent \textbf{Case (c)}\quad Type A. 6:\quad
$ \mathfrak{e}_{7}\supset \mathfrak{so}(12)\oplus \mathfrak{su}(2)
  \supset \mathfrak{so}(8)\oplus 3\mathfrak{su}(2)\supset 7\mathfrak{su}(2).$

In this case,
we have
\begin{eqnarray*}
  c_{1}&=&\frac{5}{9},\quad  c_{2}=\frac{1}{3},\quad B_{\mathfrak{su}(2)}=\frac{1}{9}B|_{\mathfrak{su}(2)},\\
   l_{\mathfrak{p}}&=&\frac{1}{2}, \quad k_{\mathfrak{p}}=\frac{5}{18}, \quad h_{\mathfrak{p}}=\frac{1}{6},\\
   \end{eqnarray*}
   and
$$  \omega_{1}=-\frac{1}{18}, \quad \omega_{2}=-\frac{2}{9}.$$
Then we have
 \begin{equation*}
   \beta=\frac{2k_{\mathfrak{p}}-2h_{\mathfrak{p}}+1-c_{2}}{c_{2}+2h_{\mathfrak{p}}}=\frac{4}{3},
 \end{equation*}
 and
  \begin{eqnarray*}
   \bar{f}(\beta) &=& \frac{1-c_{1}}{8}\beta (\omega_{1}\beta+k_{\mathfrak{p}}-h_{\mathfrak{p}}) \\
    &=& \frac{4}{9}\times \frac{1}{8}\times \frac{4}{3}(-\frac{1}{18}\times \frac{4}{3}+\frac{5}{18}-\frac{1}{6}) \\
    &=& \frac{2}{729}.
 \end{eqnarray*}
 Moreover,
 \begin{eqnarray*}
   \bar{f}(1) &=& M(1-\beta)\eta^{2}(1)+\frac{1-c_{1}}{8}(\omega_{1}+k_{\mathfrak{p}}-h_{\mathfrak{p}})\\
    &=& (\frac{1}{4}\times \frac{1}{3}+\frac{1}{2}\times \frac{1}{6})(1-\frac{4}{3})
    \frac{1}{4}(\frac{1}{2}+\frac{1}{2}-\frac{1}{6}-\frac{1}{2}\times\frac{5}{9})^{2}
    +\frac{1}{8}(1-\frac{5}{9})(-\frac{1}{18}+\frac{5}{18}-\frac{1}{6}) \\
    &=& -\frac{7}{5832}.
 \end{eqnarray*}
 Thus $1<z<\frac{4}{3}$.

\medskip
\noindent \textbf{Case (d)}\quad Type B. 3  with $n_{1}=n_{2}=2$:\quad
 $
   \mathfrak{sp}(4k)\supset 2\mathfrak{sp}(2k)\supset 4\mathfrak{sp}(k)\supset \{e\}, \quad k\geq 1.$

In this case, we have
\begin{eqnarray*}
  c_{1}&=&\frac{2k+1}{4k+1},\quad c_{2}=\frac{k+1}{4k+1}, \\
  l_{\mathfrak{p}}&=&\frac{1}{2},\quad k_{\mathfrak{p}}=\frac{2k+1}{2(4k+1)},
  \end{eqnarray*}
  and
  $$\omega_{1}=-\frac{1}{2(4k+1)},\quad \omega_{2}=\frac{2k}{4k+1}> 0.$$
 Then we have
 \begin{equation*}
   \beta=\frac{2k_{\mathfrak{p}}-2h_{\mathfrak{p}}+1-c_{2}}{c_{2}+2h_{\mathfrak{p}}}
   =\frac{2k+1+4k+1-k-1}{k+1}=\frac{5k+1}{k+1},
 \end{equation*}
 and
 \begin{eqnarray*}
   \bar{f}(\beta) &=& \frac{1-c_{1}}{8}\beta (\omega_{1}\beta+k_{\mathfrak{p}}-h_{\mathfrak{p}}) \\
    &=& \frac{2k}{4k+1}\times \frac{1}{8}\times \frac{5k+1}{k+1}
    (-\frac{1}{2(4k+1)}\times \frac{5k+1}{k+1}+\frac{2k+1}{2(4k+1)}) \\
    &=& \frac{k^{2}(5k+1)(k-1)}{4(4k+1)^{2}(k+1)^{2}}\geq 0.
 \end{eqnarray*}
 Thus $\bar{f}(\beta)=0$ if and only if $k=1$.

\medskip
 \noindent\textbf{Case (e)}\quad  Type B. 4:\quad
$
  \mathfrak{so}(8)\supset \mathfrak{so}(7)\supset \mathfrak{g}_{2}\supset \{e\}.$

In this case, we
 have
\begin{eqnarray*}
  c_{1}&=&\frac{5}{6},\quad  c_{2}=\frac{2}{3},\\
   l_{\mathfrak{p}}&=&\frac{1}{2}, \quad k_{\mathfrak{p}}=\frac{1}{3},
   \end{eqnarray*}
   and
  $$\omega_{1}=-\frac{1}{4}, \quad \omega_{2}=\frac{1}{6}.$$
Then we have
\begin{equation*}
   \beta=\frac{2k_{\mathfrak{p}}-2h_{\mathfrak{p}}+1-c_{2}}{c_{2}+2h_{\mathfrak{p}}}=\frac{3}{2},
 \end{equation*}
 and
  \begin{eqnarray*}
   \bar{f}(\beta) &=& \frac{1-c_{1}}{8}\beta (\omega_{1}\beta+k_{\mathfrak{p}}-h_{\mathfrak{p}}) \\
    &=& \frac{1}{6}\times \frac{1}{8}\times \frac{3}{2}(-\frac{1}{4}\times \frac{3}{2}+\frac{1}{3}) \\
    &=& -\frac{1}{768}.
 \end{eqnarray*}
So $z>\frac{3}{2}$.

\medskip
\noindent\textbf{Case (f)}\quad Type B. 5:\quad
$
  \mathfrak{f}_{4}\supset \mathfrak{so}(9)\supset \mathfrak{so}(8)\supset \{e\}.$

In this case, we have
\begin{eqnarray*}
  c_{1}&=&\frac{7}{9},\quad  c_{2}=\frac{2}{3},\\
   l_{\mathfrak{p}}&=&\frac{1}{2}, \quad k_{\mathfrak{p}}=\frac{7}{18},
   \end{eqnarray*}
   and
 $$ \omega_{1}=-\frac{5}{18}, \quad \omega_{2}=\frac{2}{9}.$$
Then we have
\begin{equation*}
   \beta=\frac{2k_{\mathfrak{p}}-2h_{\mathfrak{p}}+1-c_{2}}{c_{2}+2h_{\mathfrak{p}}}=\frac{5}{3},
 \end{equation*}
 and
\begin{eqnarray*}
   \bar{f}(\beta) &=& \frac{1-c_{1}}{8}\beta (\omega_{1}\beta+k_{\mathfrak{p}}-h_{\mathfrak{p}}) \\
    &=& \frac{2}{9}\times \frac{1}{8}\times \frac{5}{3}(-\frac{5}{18}\times \frac{5}{3}+\frac{7}{18}) \\
    &=& -\frac{5}{1458}.
\end{eqnarray*}
So $z>\beta=\frac{5}{3}$.

\end{document}